\documentclass[11pt, reqno]{amsart}

\usepackage{a4wide, 
mathrsfs,
amsfonts,
enumitem,
comment,
upgreek,
soul,
mathtools,
shuffle,
kotex,
array,
caption,
}
\usepackage{verbatim}
\usepackage{thmtools}
\usepackage{thm-restate}
\usepackage{amsmath,amssymb,amsthm}
\usepackage{color}
\usepackage{hyperref}
\usepackage{setspace}
\usepackage{graphicx}
\usepackage{mathtools}
\usepackage[dvipsnames]{xcolor}
\usepackage{multirow}

\usepackage[linesnumbered,ruled]{algorithm2e}
\usepackage{here}
\usepackage{pict2e}
\usepackage{mathpazo}
\usepackage{xifthen}
\usepackage[thinc]{esdiff}
\usepackage{kotex}
 
\usepackage[capitalise,nameinlink]{cleveref}
\usepackage{xifthen}
\usepackage[T1]{fontenc}
\usepackage{url}
\usepackage[all]{xy}
\usepackage{tikz,tikz-cd}
\usepackage{makecell}
\usepackage{ctable}

\crefname{equation}{}{}
\crefname{figure}{{\sc Figure}}{{\sc Figure}}
\usepackage[euler]{textgreek}
\usepackage[colorinlistoftodos, textwidth = 2.3cm]{todonotes}

\def\multiset#1#2{\ensuremath{\left(\kern-.3em\left(\genfrac{}{}{0pt}{}{#1}{#2}\right)\kern-.3em\right)}}

\hypersetup{
    colorlinks=true,
    linkcolor=blue,
    citecolor=magenta,      
    urlcolor=cyan,
}

\newtheorem{theorem}{Theorem}[section]
\newtheorem{lemma}[theorem]{Lemma}
\newtheorem{question}[theorem]{Question}
\newtheorem{corollary}[theorem]{Corollary}
\newtheorem{proposition}[theorem]{Proposition}
\theoremstyle{definition}
\newtheorem{definition}[theorem]{Definition}
\newtheorem{example}[theorem]{Example}

\theoremstyle{remark}
\newtheorem{remark}[theorem]{Remark}

\numberwithin{equation}{section} \numberwithin{figure}{section}
\numberwithin{table}{section}
\newcommand{\F}{(\mathbb{F}_q^n, \textup{Euc}_n)}

\title[Combinatorics of Euclidean Spaces over Finite Fields]
{Combinatorics of Euclidean Spaces\\ over Finite Fields}  

\author{Semin Yoo}

\date{\today}
\address{School of Computational Sciences, Korea Institute for Advanced Study (KIAS), 85 Hoegiro Dongdaemun-gu, Seoul 02455, Republic of Korea }
\email{syoo19@kias.re.kr}

\begin{document}

\maketitle

\begin{abstract}
The $q$-binomial coefficients are $q$-analogues of the binomial coefficients, counting the number of $k$-dimensional subspaces in the $n$-dimensional vector space $\mathbb{F}^n_q$ over $\mathbb{F}_{q}$. In this paper, we define a Euclidean analogue of $q$-binomial coefficients as the number of $k$-dimensional subspaces which have an orthonormal basis in the quadratic space $(\mathbb{F}_{q}^{n},x_{1}^{2}+x_{2}^{2}+\cdots+x_{n}^{2})$. We prove its various combinatorial properties compared with those of $q$-binomial coefficients. 
In addition, we formulate the number of subspaces of other quadratic types and study some related properties. 
\end{abstract}

\section{Introduction and statements of results}\label{introduction}

A \textbf{Euclidean space} is a finite-dimensional vector space over the real numbers with an inner product, i.e. a positive-definite symmetric bilinear form. This allows us to define the concept of orthonormality between vectors, in particular, providing the well-known fact that any Euclidean space has an orthonormal basis. 

In this paper, we change the base field of vector spaces from the field of real numbers to finite fields. 
Since there does not exist the concept of positivity for numbers in finite fields, we consider a symmetric bilinear form, called a \textbf{quadratic form}, ignoring the property of positive-definite of inner products. 
A \textbf{quadratic space} is a vector space over any field equipped with a quadratic form.
Throughout the paper, we assume that the characteristic of the base field of quadratic spaces is not two since the definition of quadratic forms over the fields of odd characteristic is different in the case where the characteristic of the base field is two. (See Chapter $7$ in \cite{Co}.)
One can show that any quadratic space over a finite field has an orthogonal basis, but there is no guarantee whether quadratic spaces over finite fields have an orthonormal basis unlike over real numbers. 

We restrict our interest to $\mathbb{F}_{q}^{n}$, the $n$-dimensional vector space over the finite field $\mathbb{F}_{q}$ with $q$ elements, where $q$ is an odd prime power. 
Note that the number of $k$-dimensional subspaces of $\mathbb{F}_{q}^{n}$ is given by the $q$-binomial coefficient $\binom{n}{k}_{q}$ which plays an important role in combinatorics. 
One of the surprising phenomena about the $q$-binomial coefficient $\binom{n}{k}_{q}$ is the fact that the $q$-binomial coefficient $\binom{n}{k}_{q}$ can be considered as the $q$-analogues of the binomial coefficients.

$q$-Analogues of quantities in mathematics involve perturbations of classical quantities using the parameter $q$, and revert to the original quantities when $q$ goes $1$.
The $q$-analogue of binomial coefficients is one of the important examples.
When $q$ goes to $1$, the $q$-binomial coefficient $\binom{n}{k}_{q}$ reverts to the binomial coefficient $\binom{n}{k}$ which measures the number of $k$-sets in $\left [ n \right ]$. 
We summarize well-known relationships between the binomial coefficients and the $q$-binomial coefficients in \cref{field}. 
In the table, we denote $A \in GL(k,q)$, $B \in GL(n-k,q)$, and $C$ is a $k \times (n-k)$ matrix.
\begin{table}[t]
\renewcommand{\arraystretch}{1.2}
\begin{center}
\scalebox{0.9}{
\begin{tabular}{c||c|c}
 & \textbf{Field with one element}  & $\mathbf{\mathbb{F}_{\mathit{q}}}$ \textbf{($\mathbf{\mathit{q}}$-analogues)} \\ \hline \hline
 
\textbf{object} & $\left [ n \right ]=\left \{ 1,2,\ldots,n \right \}$  & 
$\mathbb{F}^{n}_{q}$ \\ \hline
\textbf{subobject} &   a $k$-elements subset in $\left [ n \right ]$ & 
a $k$-dimensional subspace of $\mathbb{F}^{n}_{q}$ \\ \hline
\textbf{bracket}   & $n$ & the number of lines in $\mathbb{F}_{q}^{n}$ \\ \hline
\textbf{factorial} & $n!$ & $\left [ n\right ]_{q}!$ \\ \hline
\textbf{poset} & $B_{n}$ & $L_{n}(q)$ \\ \hline
\textbf{flag} & a flag in $\left [ n \right ]$ & a flag in $\mathbb{F}_{q}^{n}$\\ \hline
\textbf{group}& $|S_{n}|=n!$ & $|GL(n,q)|=q^{n(n-1)/2}(q-1)^{n}\left [ n \right ]_{q}!$ \\ \hline
\textbf{formula} & $\binom{n}{k}=\frac{n!}{k!(n-k)!}=\left | \frac{S_{n}}{S_{k}\times S_{n-k}} \right |$  & $\binom{n}{k}_{q}=\frac{[n]_{q}!}{[k]_{q}![(n-k)]_{q}!}=\Big | \frac{GL(n,q)}{\bigl(\begin{smallmatrix}
A & C \\ 
\mathbf{0} & B
\end{smallmatrix}\bigr)} \Big |$ \\ \hline
\textbf{connection}& \multicolumn{2}{c}{$
\lim_{q \rightarrow 1} \binom{n}{k}_{q}=\binom{n}{k}$}\\
\end{tabular}}
\caption{Example of Field with one element analogues}
\label{field}
\end{center}
\end{table}
It turns out that $q$-analogues appear in many mathematical areas not only combinatorics but also in the study of special functions and quantum groups. For more discussions on this topic, we refer the reader to \cite{Lo, St3}.

In a group theoretic perspective, we note that binomial coefficients are belong to the theory governed by symmetric groups since the symmetric group $S_{n}$ acts transitively on the set of $k$-elements sets in $[n]$. 
Similarly, since the general linear group $GL(n,q)$ over $\mathbb{F}_{q}$ acts transitively on the set of $k$-dimensional subspaces of $\mathbb{F}_{q}^{n}$, the theory of $q$-binomial coefficients is governed by general linear groups.

The main purpose of this paper is to introduce a formula of counting the number of subspaces of $(\mathbb{F}_{q}^{n},x_{1}^{2}+x_{2}^{2}+\cdots+x_{n}^{2})$ that have an orthonormal basis, which can be written as an analogue of binomial coefficients (\cref{subsp} and \cref{polynomials}), and to study related combinatorics listed in the last column of \cref{table-analog}. 
We develop a new analogue of binomial coefficients which is governed by the orthogonal group $O(n,q)$ over $\mathbb{F}_{q}$.

\begin{table}[H]
\renewcommand{\arraystretch}{1.3}
\begin{center}
\scalebox{0.87}{
\begin{tabular}{c||c|c}
 &  $\mathbf{\mathit{q}}$\textbf{-analogues} & \textbf{Euclidean-analogues} \\ \hline \hline
\textbf{space} &  $\mathbb{F}^{n}_{q}$ & 
$(\mathbb{F}^{n}_{q},\text{Euc}_{n})$ \\ \hline
\textbf{subspace} &  a $k$-dimensional subspace of $\mathbb{F}^{n}_{q}$ & 
a $k$-dimensional Euclidean subspace of $(\mathbb{F}^{n}_{q},\text{Euc}_{n})$\\ \hline
\textbf{bracket}   & the number of lines in $\mathbb{F}^{n}_{q}$ & the number of Euclidean lines in $(\mathbb{F}^{n}_{q},\text{Euc}_{n})$  \\ \hline
\textbf{factorial} & $\left [ n\right ]_{q}!$ & $\left [n \right ]_{q}^{\perp}!$ \\ \hline
\textbf{poset} & $L_{n}(q)$ & $E_{n}(q)$ \\ \hline
\textbf{flag} & a flag in $\mathbb{F}_{q}^{n}$ & a Euclidean flag in ($\mathbb{F}_{q}^{n}$,Euc$_{n}$)\\ \hline
\textbf{group}& $|GL(n,q)|=q^{n(n-1)/2}(q-1)^{n}\left [n \right]_{q}!$ & $|O(n,q)|=2^{n}\left [n\right ]_{q}^{\perp}!$ \\ \hline
\textbf{formula} &  $\binom{n}{k}_{q}=\frac{[n]_{q}!}{[k]_{q}![n-k]_{q}!}=\left | \frac{GL(n,q)}{\bigl(\begin{smallmatrix}
A & C \\ 
\mathbf{0} & B
\end{smallmatrix}\bigr)} \right |$ & $\binom{n}{k}_{q}^{\perp}=\frac{[n]_{q}^{\perp}!}{[k]_{q}^{\perp}![n-k]_{q}^{\perp}!}=\left | \frac{O(n,q)}{O(k,q)\times O(n-k,q)} \right |$ \\ \hline
\textbf{polynomial} & a polynomial degree of $k(n-k)$ in $q$  & a polynomial degree of $k(n-k)$ in $q$ \\ \hline
 \multirow{2}{*}{\textbf{limit}} & \multirow{2}{*}{$\lim_{q \to 1}\binom{n}{k}_{q}=\binom{n}{k}$} & $\lim_{q \to 1}\binom{n}{k}_{q}^{\perp}=\binom{\left \lfloor n/2 \right \rfloor}{\left \lfloor k \right \rfloor}$ when $q \equiv$ 1 (mod $4$)\\
 & & $\lim_{q \to -1}\binom{n}{k}_{q}^{\perp}=\binom{\left \lfloor n/2 \right \rfloor}{\left \lfloor k \right \rfloor}$ when $q \equiv$ 3 (mod $4$)
\end{tabular}}
\caption{The $q$-analogues and the Euclidean-analogues}
\label{table-analog}
\end{center}
\end{table} 

From here, we elaborate on our work. 
Recall that in the theory of quadratic forms, any quadratic form $Q$ on a vector spaces $V$ over $F$ of dimension $n$ can be diagonalized if char$(F)\ne 2$. In other words, after changing coordinates, $Q$ can be written as 
$Q=\sum_{i=1}^{n}a_{i}x_{i}^{2}$ for some $a_{i} \in F$. Here, $a_{i}$ could be zero. To find an explicit form of $Q$, we only need to determine the coefficients $a_{i}$'s. 
For example, if $F=\mathbb{R}$ (or $F=\mathbb{C}$), any non-degenerate quadratic form on a vector space of dimension $n$ is equivalent to 
$x_{1}^{2}+x_{2}^{2}+\cdots+x_{p}^{2}-x_{p+1}^{2}-x_{p+2}^{2}-\cdots -x_{n}^{2}$ for a unique $0 \leq p \leq n$  (or $x_{1}^{2}+x_{2}^{2}+\cdots+x_{n}^{2}$, respectively).
Over $\mathbb{F}_{q}$, where $q$ is an odd prime power, any non-degenerate quadratic form on a vector space of dimension $n$ is equivalent to one of the following two forms:
\begin{align}\label{classi1}
\begin{split}
\text{Euc}_n(\mathbf{x}) & :=x_{1}^{2}+x_{2}^{2}+\cdots+x_{n-1}^{2}+x_{n}^{2} \quad \text{or} \\
\text{Lor$_{n}$}(\mathbf{x}) & :=x_{1}^{2}+x_{2}^{2}+\cdots+x_{n-1}^{2}+\lambda x_{n}^{2} \quad \text{for some non-square $\lambda$ in $\mathbb{F}_{q}$}. 
\end{split}
\end{align}
Even though this classification has been already used by number theorists as seen in \cite{Se}, we introduce some terminologies for convenience. We call $(\mathbb{F}_q^n, \text{Euc}_n)$ the \textbf{Euclidean space} over $\mathbb{F}_{q}$. 
The motivation for the name is due to the fact that any positive-definite symmetric form over real numbers (inner product) is equivalent to $x_{1}^2+x_{2}^{2}+\cdots+x_{n}^{2}$ as a quadratic form. 
We also call $(\mathbb{F}_{q}^{n}, \text{Lor}_{n})$ the \textbf{Lorentzian space} over $\mathbb{F}_{q}$, following the convention that $(\mathbb{R}^{n},x_{1}^{2}+x_{2}^{2}+\cdots+x_{n-1}^{2}-x_{n}^{2})$ is called a Lorentzian space. 
We consider the concept of negative real numbers as analogous to the concept of non-square numbers in finite fields. When there is no danger of confusion, we let 
Euc$_{n}$ and Lor$_{n}$ denote $\F$ and $(\mathbb{F}_q^n, \text{Lor}_n)$, respectively.

Note that finite geometers use the following classification of non-degenerate quadratic forms on a $n$-dimensional vector space over $\mathbb{F}_{q}$, which reveals their own geometric structures. 
A $n$-dimensional vector space over $\mathbb{F}_{q}$ with a non-degenerate quadratic form is isometrically isomorphic to one of the following spaces:
\begin{align}\label{classi2}
\begin{split}
& \textup{hyperbolic}:~  k\mathbb{H} \textup{ when }n=2k \textup{ and } \mathbb{H} \textup{ is the hyperbolic plane},\\
& \textup{elliptic}:~ (k-1)\mathbb{H} \oplus (x^{2}-\lambda y^{2}) \textup{ when }n=2k  \textup{ and }\lambda \textup{ is a non-square}, \\
& \textup{parabolic}:~ k\mathbb{H}\oplus cx^{2} \textup{ when } n=2k+1 \textup{ and } c \textup{ is either }1 \textup{ or a non-square }. 
\end{split}
\end{align}
For the definition of the hyperbolic plane, see \cref{sec2}.
We are interested in using the first number theoretic classification. 
It is possible to recover mathematical results from \cref{classi1} to the ones in the version of \cref{classi2} by matching the correct quadratic type between \cref{classi1} and \cref{classi2}.
For example, one can see \cref{corresp}.

Let us call a $k$-dimensional subspace $W$ of $(\mathbb{F}_{q}^{n},\text{Euc}_{n})$ \textbf{Euclidean} (or \textbf{Lorentzian}, respectively) if $(W,\text{Euc}_{n}|_{W})$ is isometrically isomorphic to $(\mathbb{F}_{q}^{k},\text{Euc}_{k})$ (or $(\mathbb{F}_{q}^{k},\text{Lor}_{k}) $, respectively) where Euc$_{n}|_{W}$ is the restricted quadratic form on $W$. We will call a $1$-dimensional Euclidean (or Lorentzian) subspace of $(\mathbb{F}_{q}^{n},\text{Euc}_{n})$ a \textbf{Euclidean line} (or \textbf{Lorentzian line}, respectively) of $(\mathbb{F}_{q}^{n},\text{Euc}_{n})$. Note that $(W,\text{Euc}_{n}|_{W})$ is a $k$-dimensional Euclidean subspace of $(\mathbb{F}_{q}^{n},\text{Euc}_{n})$ if and only if $(W,\text{Euc}_{n}|_{W})$ has an orthonormal basis. 
Thus, in order to count the number of $k$-dimensional subspaces of $(\mathbb{F}_{q}^{n},\text{Euc}_{n})$ that have an orthonormal basis, we count the number of $k$-dimensional Euclidean subspaces of $(\mathbb{F}_{q}^{n},\text{Euc}_{n})$.

In \cref{section3}, we consider the poset of all Euclidean subspaces of $(\mathbb{F}_{q}^{n},\text{Euc}_{n})$ under set-inclusion,  
which we call the \textbf{Euclidean poset} and denote it by $E_n(q)$. In \cite{St2}, the author shows that the poset $L_n(q)$ which consists of all subspaces of  $\mathbb{F}_{q}^{n}$ under set-inclusion is rank symmetric and rank unimodal. We show that the Euclidean poset $E_{n}(q)$ is rank symmetric (\cref{rs}) and rank unimodal (\cref{ru}). Furthermore, using the structure of the Euclidean poset $E_n(q)$, we find that the number of $k$-dimensional Euclidean subspaces of $(\mathbb{F}_{q}^{n},\text{Euc}_{n})$ is given by
\[
|\text{Euc}_{k},\text{Euc}_{n}|_{q}:=\frac{|\text{Euc}_{1},\text{Euc}_{n}|_{q}|\text{Euc}_{1},\text{Euc}_{n-1}|_{q}\cdots|\text{Euc}_{1},\text{Euc}_{n-k+1}|_{q}}{|\text{Euc}_{1},\text{Euc}_{k}|_{q}\cdots|\text{Euc}_{1},\text{Euc}_{1}|_{q}}.
\]
Observing the similarity between this formula and the one for the binomial coefficients, we define the \textbf{Euclidean binomial coefficient} (\cref{dot-binom}) as follows:
\[
\binom{n}{k}_{q}^{\perp}:=|\text{Euc}_{k},\text{Euc}_{n}|_{q}=\frac{[n]_{q}^{\perp}!}{[k]_{q}^{\perp}![n-k]_{q}^{\perp}!},
\]
where $[k]_{q}^{\perp}=|\text{Euc}_{1},\text{Euc}_{k}|_{q}$ and $[n]_{q}^{\perp}! =\prod_{k=1}^n [k]_{q}^{\perp}$. The bracket $[k]_{q}^{\perp}$, the number of Euclidean lines in $(\mathbb{F}_{q}^{k},\text{Euc}_{k})$, is given in \cref{dotline}. 
It is worthwhile to note that $[k]_{q}^{\perp}$ depends on the number theoretic condition of whether $q \equiv 1$ (mod $4$) or $q \equiv 3$ (mod $4$), and so do the Euclidean binomial coefficients.
This condition is equivalent to the case when $-1$ is a square or a non-square, which is required due to the fact that the discriminant of the hyperbolic plane is $-1$. More details can be found in the first proof of \cref{dotline}. 
For a computational argument, see the second proof of \cref{dotline}.
Additionally, we verify that the line counts of each type gives a new isometric invariant of quadratic spaces over finite fields (\cref{invariant}).

We refer to the reader that related combinatorial applications of the Euclidean binomial coefficients can be found in \cite{Yo2} for the associated pseudo-random graphs constructed by Euclidean subspaces of $(\mathbb{F}_{q}^{n},\text{Lor}_{n})$, and \cite{Yo3} for the incidence graphs between Euclidean subspaces of $(\mathbb{F}_{q}^{n},\text{Euc}_{n})$.
We also state some relevant counts from finite geometry.
The study of totally isotropic spaces over finite fields has been conducted by a considerable number of researchers for long time ago. 
(For the definition of isotropic spaces, see \cref{sec2}.)
For example, the number of  $k$-dimensional totally isotropic subspaces of a non-degenerate quadratic space $V$ is well-known. 
In addition, the set of maximal singular subspaces of $V$ constructs interesting graphs, which are distance-transitive, and related parameters can be computed. More details on this topic can be found in \cite{Mu, Sta}.

Next, we investigate various combinatorial properties of $\binom{n}{k}_{q}^{\perp}$, including the analogues of Pascal's identity  (\cref{pa} and \cref{pa2}) and log-concavity (\cref{lc}). We also show that the Euclidean binomial coefficient $\binom{n}{k}_{q}^{\perp}$ is expressed in terms of the product of $q$-binomial coefficients and some polynomials (\cref{polynomials}). For example, when $q \equiv 1$ (mod $4$), and $n,k$ are odd, it is written as
\[\binom{n}{k}_{q}^{\perp}=\frac{1}{2}q^{\frac{k(n-k)}{2}}(q^{\frac{n-k}{2}}+1)\binom{\frac{n-1}{2}}{\frac{k-1}{2}}_{q^{2}},\]
where \[\binom{n}{k}_{q^{2}}=\frac{((q^{2})^{n}-1)((q^{2})^{n-1}-1)\cdots((q^{2})^{n-k+1}-1)}{((q^{2})^{k}-1)((q^{2})^{k-1}-1)\cdots(q^{2}-1)}.\]
Furthermore, we prove that $\binom{n}{k}_{q}^{\perp}$ is a polynomial of degree $k(n-k)$ in $\frac{1}{2}\mathbb{Z}[q]$ (\cref{rational}) and the coefficients of the polynomial $\binom{n}{k}_{q}^{\perp}$ have some symmetry (\cref{symmetry}). For example, when $q \equiv 1$ (mod $4$), and $n,k$ are odd, the Euclidean binomial coefficients can be written as
\[\binom{n}{k}_{q}^{\perp}=q^{\frac{k(n-k)}{2}}\cdot \sum_{i=0}^{\frac{k(n-k)}{2}}a_{i}q^{i},\]
and we have the property $a_{i}=a_{k(n-k)/2-i}$ for $i=0,1,\ldots,k(n-k)/2$. 
In addition, we find the corresponding set theoretic analogue to $\binom{n}{k}_{q}^{\perp}$ on the set side in the following sense.
Since taking limits for $\binom{n}{k}_{q}$ as $q$ goes to $1$ reveals the connection between $k$-subsets of $[n]$ and $k$-dimensional subspaces of $\mathbb{F}_{q}^{n}$, it is reasonable to take limits for $\binom{n}{k}_{q}^{\perp}$ as $q$ goes to $1$. 
In Euclidean-analogues, we find the combinatorial correspondence between symmetric $k$-subsets of $\mathbb{Z}/(n+1)\mathbb{Z}$ and $k$-dimensional Euclidean subspaces of $\F$ taking limits for $\binom{n}{k}_{q}^{\perp}$ as $q$ goes to $1$ when $q\equiv 1$ (mod $4$) and as $q$ goes $-1$ when $q\equiv 3$ (mod $4$) instead of taking limits as $q$ goes to $1$ when $q \equiv 3$ (mod $4$) (\cref{symm}).

We also discuss how to compute the size of the orthogonal group $O(n,q)$ over $\mathbb{F}_{q}$ with Euc$_{n}$. Recall that the size of symmetric group $S_{n}$ is $n!$ and the size of the general linear group $GL(n,q)$ over $\mathbb{F}_{q}$ can be written by 
\[|GL(n,q)|=q^{n(n-1)/2}(q-1)^{n}\left [n \right]_{q}!.\]
Similarly, by using the number of flags of the Euclidean poset $E_{n}(q)$, the size of orthogonal group $O(n,q)$ over $\mathbb{F}_{q}$ (\cref{sizeofog}) can be given by 
\[|O(n,q)|=2^{n}\left [ n \right ]_{q}^{\perp}!,\]
and we also show that 
\[
\binom{n}{k}_{q}^{\perp}=\left | \frac{O(n,q)}{O(k,q)\times O(n-k,q)} \right |.
\]
This implies the following observation:
\[Gr_{q}^{\perp}(k,n)=\frac{O(n,q)}{O(k,q)\times O(n-k,q)} \subsetneq Gr_{q}(k,n)=\frac{GL(n,q)}{\bigl(\begin{smallmatrix}
A & C \\ 
\mathbf{0} & B
\end{smallmatrix}\bigr)},\]
where $Gr_{q}^{\perp}(k,n)$ is the set of $k$-dimensional Euclidean subspaces of $(\mathbb{F}_{q}^{n},\text{Euc}_{n})$, $Gr_{q}(k,n)$ is the set of $k$-dimensional subspaces of $\mathbb{F}_{q}^{n}$, $A \in GL(k,q)$ and $B \in GL(n-k,q)$ and $C$ is a $k \times (n-k)$ matrix. This is a different phenomenon from the real case since the set of $k$-dimensional subspaces of a vector space over real numbers, the real Grassmannian, is $O(n,\mathbb{R})/O(k,\mathbb{R})\times O(n-k,\mathbb{R})$. This is because it is not true in general that there exists an orthonormal basis of quadratic spaces over finite fields. 
In the last place of the section, we discuss how to compute the M\"{o}bius function of the Euclidean poset $E_{n}(q)$ (\cref{mobi}).

In \cref{section4}, we consider another non-degenerate quadratic form, the Lorentzian form, Lor$_{n}$. 
Let us call a $k$-dimensional subspace $W$ of $(\mathbb{F}_{q}^{n},\text{Euc}_{n})$ \textbf{Lorentzian} if $(W,\text{Euc}_{n}|_{W})$ is isometrically isomorphic to $(\mathbb{F}_{q}^{k},\text{Lor}_{k})$, where Euc$_{n}|_{W}$ is the restricted quadratic form on $W$. 
We provide formulas that count the number of $k$-dimensional Euclidean subspaces and Lorentzian subspaces of $(\mathbb{F}_{q}^{n},\text{Euc}_{n})$, and $(\mathbb{F}_{q}^{n}, \text{Lor}_{n})$, respectively (\cref{other counts}).  
Let us denote them by $\binom{n}{k}_{q}^{\perp}, \binom{n}{\overline{k}}_{q}^{\perp}, \binom{\overline{n}}{k}_{q}^{\perp}$, and $\binom{\overline{n}}{\overline{k}}_{q}^{\perp}$, respectively.
We show that $\binom{n}{\overline{k}}_{q}^{\perp}$ is a polynomial in $q$ whose coefficients are half-integers (\cref{polynomial2}).
Furthermore, we verify that $\binom{n}{k}_{q}^{\perp}+\binom{n}{\overline{k}}_{q}^{\perp}$ is a polynomial with non-negative integer coefficients except for some cases (\cref{prop: nondeg}).
We finish this section by proving that the Lorentzian poset $LO_{n}(q)$, consisting of all Lorentzian subspaces of $\F$ under set-inclusion, is rank-symmetric (\cref{prop: LO1}) and rank-unimodal (\cref{prop: LO2}).

Last but not the least, we discuss a possible direction of future research from the Euclidean binomial coefficient in \cref{Future}.

\bigskip
\noindent \textbf{Acknowledgements.} This work is a part of the author's Ph.D. dissertation. The author would like to express her great gratitude to her thesis advisor Jonathan Pakianathan for helpful discussions and encouragement for this work. The author also would like to thank Harry Richman for pointing out an error in an initial draft, Ferdinand Ihringer for helpful suggestions on this work, IBS Discrete Mathematics Group for suggesting better terminologies for this work, and Mark Bly for suggesting to define multibinomial coefficients. 
Lastly, the author appreciates the referee's valuable comments on this paper, which significantly improved the quality of the paper.
The author is supported by the KIAS Individual Grant (CG082701) at Korea Institute for Advanced Study.

\section{Review of the theory of quadratic forms}\label{sec2}

In this section, we briefly remind the reader of some basic concepts of quadratic forms that we use later, and mainly follow \cite{Cl, Co, Se}. Readers familiar with the theory of quadratic forms can skip this section.

Let $V$ be a $n$-dimensional vector space over a field $F$ with char$(F)\ne 2$. A \textbf{quadratic form} $Q$ on $V$ is a function from $V$ to $F$ satisfying the following two conditions: $Q(cv)=c^{2}Q(v)$ for any $v \in V,c\in F$ and  $B(v,w)=\frac{1}{2}(Q(v+w)-Q(v)-Q(w))$ is bilinear. 
We call $B$ the \textbf{bilinear form associated with} $Q$ and dim$V$ the \textbf{dimension of the quadratic form} $Q$. 
It is easy to check that $B$ is symmetric. 
Once we fix a basis of $V$, quadratic forms on $V$ can be expressed by a unique matrix form. We call such a matrix $M$ the \textbf{matrix associated with }$Q$ in this basis. Furthermore, there are canonical bijections among the following sets in a chosen basis:
\renewcommand\labelitemi{\tiny$\bullet$}
\begin{itemize}
\item The set of homogeneous polynomials of degree $2$, in $n$-variables.
\item The set of quadratic forms on $V$.
\item The set of symmetric bilinear forms on $V$.
\item The set of symmetric $n\times n$ matrices over $F$.
\end{itemize}
We will use any of these definitions of quadratic forms as occasion demands.

Next, we define a special class of quadratic forms, which forms the building blocks of quadratic forms. A quadratic form is called \textbf{non-degenerate} if a matrix representation of it is invertible. If its determinant is zero, we call a quadratic form \textbf{degenerate}. Two quadratic forms $Q_{1}, Q_{2}$ on $V$ are \textbf{equivalent} if there is a linear isomorphism $A:V\longrightarrow V$ such that $Q_{2}(Av)=Q_{1}(v)$ for any $v$ in $V$. For example, the two quadratic forms  $Q(x,y)=x^{2}-y^{2}~~\text{and}~~Q'(x,y)=xy$ on $\mathbb{R}^{2}$ are equivalent by $A=\bigl(\begin{smallmatrix}
1 & 1\\ 
1 & -1
\end{smallmatrix}\bigr)$.

A natural problem is to determine the classification of quadratic forms up to equivalence. We restrict the base field to a finite field. 
In \cref{introduction}, we already considered two types of the classification of non-degenerate quadratic forms over finite fields up to equivalence: \cref{classi1} and \cref{classi2}.
The proofs can be found in \cite{Co}.

The \textbf{discriminant} $d(Q)$ of $Q$ is the coset of det$(M)$ in $\mathbb{F}_{q}^{*}/\mathbb{F}_{q}^{*2}$, where $M$ is a matrix associated with $Q$. We have the following corollary directly by the classification \cref{classi1}.  
\begin{corollary}\label{coro}\cite{Se}
Two non-degenerate quadratic forms over $\mathbb{F}_{q}$ are equivalent if and only if they have the same dimension and same discriminant.
\end{corollary}

Note that the cases in the classification \cref{classi2} correspond to one or more of these cases in the classification \cref{classi1}.
This can be determined by considering their discriminants.
In the following proposition, we observe when parabolic spaces are equivalent to either the Euclidean space or the Lorentzian space.

\begin{proposition}\cite{Yo2}\label{corresp} 
Let $k$ be a positive integer. 
\begin{enumerate}
\item[\textup{(1)}] $k\mathbb{H} \oplus x^{2}$ is equivalent to \textup{Euc}$_{2k+1}$ if $q \equiv 1 \pmod{4}$ , or $q \equiv 3 \pmod{4}$ and $k$ is even,
\item[\textup{(2)}] $k\mathbb{H} \oplus x^{2}$ is equivalent to \textup{Lor}$_{2k+1}$ if $q \equiv 3 \pmod{4}$ and $k$ is odd,
\item[\textup{(3)}] $k\mathbb{H} \oplus \lambda x^{2}$ is equivalent to \textup{Euc}$_{2k+1}$ if $q \equiv 3 \pmod{4}$ and $k$ is odd,
\item[\textup{(4)}] $k\mathbb{H}\oplus \lambda x^{2}$ is equivalent to \textup{Lor}$_{2k+1}$ if $q \equiv 1 \pmod{4}$, or $q \equiv 3 \pmod{4}$ and $k$ is even.
\end{enumerate}
\end{proposition}

We now consider a quadratic form with a space. A \textbf{quadratic space} $(V,Q)$ is a vector space equipped with a quadratic form $Q$ on $V$. We call it \textbf{non-degenerate} when the underlying quadratic form is non-degenerate. Two quadratic spaces $(V_{1},Q_{1})$ and $(V_{2},Q_{2})$ are called  \textbf{isometrically isomorphic} if there is a linear isomorphism $A:V_{1} \longrightarrow V_{2}$ such that $Q_{2}(Av)=Q_{1}(v)$ for any $v$ in $V_{1}$. In this case, we call the map $A$ an \textbf{isometry}. When the map $A$ is injective, we say $A$ is an \textbf{isometric embedding}. The \textbf{category of quadratic spaces} over $F$ has as its objects quadratic spaces and morphisms isometric embeddings between quadratic spaces. Let $W$ be a vector subspace of $(V,Q)$ and $Q|_{W}$ be the restriction of $Q$ to $W$. One can check that the inclusion $(W,Q|_{W})$ in $(V,Q)$ is an isometric embedding. Thus, we say that $(W,Q|_{W})$ is a \textbf{quadratic subspace} of $(V,Q)$. In addition, given two quadratic spaces $(V_{1},Q_{1})$ and $(V_{2},Q_{2})$, one can construct a new quadratic space $(V_{1} \oplus V_{2},Q|_{V_{1}\oplus V_{2}})$ with $Q|_{V_{1}\oplus V_{2}}(v_{1},v_{2}):=Q|_{V_{1}}(v_{1})+Q|_{V_{2}}(v_{2})$. 

\begin{example}\label{all2}
Let us consider $W=\left \langle (1,0,0),(0,1,1) \right \rangle$ in $(\mathbb{F}_{3}^{3},\text{Euc}_{3}) $. Here, $\left \langle v_{1},v_{2},\ldots,v_{n} \right \rangle$ denotes the vector space spanned by $v_{1},v_{2},\ldots,v_{n}$. We show that $(W,\text{Euc}_{3}|_{W})$ is a $2$-dimensional Lorentzian subspace of $(\mathbb{F}_{3}^{3},\text{Euc}_{3}) $ and provide the description of all $2$-dimensional quadratic subspaces of $(\mathbb{F}_{3}^{3},\text{Euc}_{3}) $ in \cref{the description of all 2-dim}.

 For $e_{1}=(1,0,0), e_{2}=(0,1,1)$, and $\beta=\left\{ e_{1},e_{2} \right\}$, we consider the matrix associated $[\text{Euc}_{3}|_{W}]_{\beta}$ with $B:=\text{Euc}_{3}|_{W}$ in the basis $\beta$:
\begin{align*}
[\text{Euc}_{3}|_{W}]_{\beta}&=\begin{pmatrix}
B(e_{1},e_{1}) & B(e_{1},e_{2})/2\\ 
B(e_{2},e_{1})/2 & B(e_{2},e_{2})
\end{pmatrix}\\&=\begin{pmatrix}
(1,0,0)\cdot (1,0,0) & (1,0,0)\cdot (0,1,1)/2\\ 
(0,1,1) \cdot (1,0,0)/2 & (0,1,1)\cdot (0,1,1)
\end{pmatrix}
=\begin{pmatrix}
1 & 0\\ 
0 & 2
\end{pmatrix}.
\end{align*}
Since disc($\text{Euc}_{3}|_{W})=2$ and $2$ is a non-square in $\mathbb{F}_{3}$, $(W,\text{Euc}_{3}|_{W})$ is a $2$-dimensional Lorentzian subspace of $(\mathbb{F}_{3}^{3},\text{Euc}_{3}) $ by \cref{coro} and \cref{classi1}. 
Similarly, one can verify the types of other $2$-dimensional quadratic subspaces of $(\mathbb{F}_{3}^{3},\text{Euc}_{3})$. Here is the description of all $2$-dimensional subspaces $p_{i}$ of $\mathbb{F}_{3}^{3}$ for $i=1,2,\ldots,13$ using lines in them and omit the zero vector in each $p_{i}$.
\begin{table}[H]
\renewcommand{\arraystretch}{1.1}
\begin{center}
\scalebox{0.9}{
\begin{tabular}{c||c|c}
 & \textbf{Lines in the planes $p_{i}$} &  \textbf{Types of the planes $p_{i}$} \\ \hline \hline
$p_{1}$ & $\left \{ \left \langle (1,0,0) \right \rangle,\left \langle (0,1,0) \right \rangle, \left \langle (1,1,0) \right \rangle, \left \langle (1,2,0) \right \rangle \right \}$  & \multirow{3}{*}{$\text{Euclidean}$}  \\ \cline{1-2}
$p_{2}$ & $\left \{ \left \langle (1,0,0) \right \rangle,\left \langle (0,0,1) \right \rangle, \left \langle (1,0,1) \right \rangle, \left \langle (1,0,2) \right \rangle \right \}$  &  \\ \cline{1-2}
$p_{3}$ & $\left \{ \left \langle (0,1,0) \right \rangle,\left \langle (0,0,1) \right \rangle, \left \langle (0,1,1) \right \rangle, \left \langle (0,1,2) \right \rangle \right \}$  &  \\ \hline
$p_{4}$ & $\left \{ \left \langle (1,0,0) \right \rangle,\left \langle (0,1,1) \right \rangle, \left \langle (1,1,1) \right \rangle, \left \langle (1,2,2) \right \rangle \right \}$  & \multirow{6}{*}{$\text{Lorentzian}$}  \\ \cline{1-2}
$p_{5}$ & $\left \{ \left \langle (1,0,0) \right \rangle,\left \langle (0,1,2) \right \rangle, \left \langle (1,1,2) \right \rangle, \left \langle (1,2,1) \right \rangle \right \}$  &  \\ \cline{1-2}
$p_{6}$ & $\left \{ \left \langle (0,1,0) \right \rangle,\left \langle (1,0,1) \right \rangle, \left \langle (1,1,1) \right \rangle, \left \langle (1,2,1) \right \rangle \right \}$  &  \\ \cline{1-2}
$p_{7}$ & $\left \{ \left \langle (0,1,0) \right \rangle,\left \langle (1,0,2) \right \rangle, \left \langle (1,1,2) \right \rangle, \left \langle (1,2,2) \right \rangle \right \}$  & \\ \cline{1-2}
$p_{8}$ & $\left \{ \left \langle (0,0,1) \right \rangle,\left \langle (1,1,0) \right \rangle, \left \langle (1,1,1) \right \rangle, \left \langle (1,1,2) \right \rangle \right \}$  &  \\ \cline{1-2}
$p_{9}$ & $\left \{ \left \langle (0,0,1) \right \rangle,\left \langle (1,2,0) \right \rangle, \left \langle (1,2,1) \right \rangle, \left \langle (1,2,2) \right \rangle \right \}$  &  \\ \hline
$p_{10}$ & $\left \{ \left \langle (1,0,1) \right \rangle,\left \langle (0,1,1) \right \rangle, \left \langle (1,1,2) \right \rangle, \left \langle (1,2,0) \right \rangle \right \}$  & \multirow{4}{*}{$\text{degenerate}$}  \\ \cline{1-2}
$p_{11}$ & $\left \{ \left \langle (1,0,1) \right \rangle,\left \langle (1,1,0) \right \rangle, \left \langle (1,2,2) \right \rangle, \left \langle (0,1,2) \right \rangle \right \}$  &  \\ \cline{1-2}
$p_{12}$ & $\left \{ \left \langle (1,1,0) \right \rangle,\left \langle (0,1,1) \right \rangle, \left \langle (1,2,1) \right \rangle, \left \langle (1,0,2) \right \rangle \right \}$  &  \\ \cline{1-2}
$p_{13}$ & $\left \{ \left \langle (1,1,1) \right \rangle,\left \langle (1,2,0) \right \rangle, \left \langle (1,0,2) \right \rangle, \left \langle (0,1,2) \right \rangle \right \}$  &  \\ 
\end{tabular}}
\caption{The description of all $2$-dimensional quadratic subspaces of $\mathbb{F}_{3}^{3}$}
\label{the description of all 2-dim}
\end{center}
\end{table}
\end{example}

Let $(V,Q)$ be a quadratic space and $B$ be the bilinear form associated with $Q$. We say two subspaces $W_{1},W_{2}$ are \textbf{orthogonal} $W_{1} \perp W_{2}$ if we have $Q(w_{1}+w_{2})=Q(w_{1})+Q(w_{2})$ for any $w_{1}$ in $W_{1}$ and $w_{2}$ in $W_{2}$, equivalently, $B(w_{1},w_{2})=0$. Let $W$ be a subspace of $(V,Q)$. We define the \textbf{orthogonal complement}
\[W^{\perp}=\left \{ v \in V~|~B(v,w)=0~\text{for any }w\in W \right \}.\]  
One can show that $\text{dim}W+\text{dim}W^{\perp}=\text{dim}V$, and $(W^{\perp})^{\perp}=W$, which implies that taking $\perp$ is bijective.
In addition, we state some useful facts about non-degenerate quadratic spaces using the orthogonal complement. 
The following are equivalent: (1) $W$ is non-degenerate, (2) $W \cap W^{\perp}=0$,  (3) $W^{\perp}$ is non-degenerate, and (4) $V=W \oplus W^{\perp}$.

We define special classes of non-degenerate quadratic spaces. Let $(V,Q)$ be a non-degenerate quadratic space. A vector space $V$ is said to be \textbf{isotropic} if there exists a nonzero vector $v$ in $V$ satisfying $Q(v)=0$, and otherwise \textbf{anisotropic}. 
Let us call a subspace $W$ of $(V,Q)$ \textbf{totally isotropic}  if $Q|_{W}= 0$. 
We next introduce special non-degenerate quadratic spaces. 
The \textbf{hyperbolic plane} $\mathbb{H}$ is a $2$-dimensional quadratic space where the quadratic form is equivalent to $xy$ (or $x^{2}-y^{2}$). 
A quadratic space is \textbf{hyperbolic} if it is isometrically isomorphic to a direct sum of hyperbolic planes.

Finally, we state the fundamental results in the algebraic theory of quadratic forms. There are two equivalent statements of Witt's result: one is \textbf{Witt's cancellation theorem} and the other is \textbf{Witt's extension theorem}. We omit the proofs. They can be found in \cite{Cl}.

\begin{theorem}[Witt's cancellation theorem]\label{wc}\cite{Cl}
Let $U_{1},U_{2},V_{1},V_{2}$ be quadratic spaces where $V_{1}$ and $V_{2}$ are isometrically isomorphic. If $U_{1}\oplus V_{1} \cong U_{2} \oplus V_{2}$, then $U_{1} \cong U_{2}$. Here, $\cong$ means isometrically isomorphic of quadratic spaces.
\end{theorem}

\begin{theorem}[Witt's extension theorem]\label{we}\cite{Cl}
Suppose $X_{1} \cong X_{2}$, where $X_{1}=U_{1} \oplus V_{1}$, and $X_{2}=U_{2}\oplus V_{2},$ and suppose that $f:V_{1} \longrightarrow V_{2}$ an isometry. Then, there is an isometry $F:X_{1}\longrightarrow X_{2}$ such that $F|_{V_{1}}=f$ and $F(U_{1})=U_{2}$.
\end{theorem}

\section{The Euclidean poset $E_{n}(q)$ and the Euclidean binomial coefficient $\binom{n}{k}_{q}^{\perp}$}\label{section3}
In this section, we study various combinatorial properties of the Euclidean poset and the Euclidean binomial coefficients. 

\subsection{Rank symmetry, unimodality of $E_{n}(q)$, and log-concavity of $\binom{n}{k}_{q}^{\perp}$}

Let us define a poset $E_{n}(q):=(\mathcal{E},\subset)$, where $\mathcal{E}$ is the set of all Euclidean subspaces of $(\mathbb{F}_{q}^{n},\text{Euc}_{n})$ and we call it the \textbf{Euclidean poset} of rank $n$. In $E_{n}(q)$, we do not consider the empty set to be a subspace, but we consider the zero space as the least element of $E_{n}(q)$. 

\begin{example} Let us consider Euclidean subspaces in $(\mathbb{F}_{3}^{3},\text{Euc}_{3})$. There are three Euclidean lines $\ell_{1}=\left \langle (1,0,0) \right \rangle, ~\ell_{2}=\left \langle (0,1,0) \right \rangle, ~\ell_{3}=\left \langle (0,0,1) \right \rangle$ and three Euclidean planes $P_{1}=\left \langle (1,0,0),(0,1,0) \right \rangle,$ $P_{2}=\left \langle (1,0,0),(0,0,1) \right \rangle, ~P_{3}=\left \langle (0,1,0),(0,0,1) \right \rangle$. 
The formulas for the number of Euclidean lines and planes are discussed in \cref{subsp} and \cref{dotline}. 
We have the Hasse diagram of the Euclidean poset $E_{3}(3)$ in \cref{pic: Hasse}, which is isomorphic to the Boolean algebra $B_{3}$.
\end{example}

Recall that $|\text{Euc}_{k},\text{Euc}_{n}|_{q}$ is the number of $k$-dimensional Euclidean subspaces of $(\mathbb{F}_{q}^{n},\text{Euc}_{n})$.
In \cref{subsp}, we will show that $|\text{Euc}_{k},\text{Euc}_{n}|_{q}$ can be written as an analogue of binomial coefficients, and thus we will denote it by $\binom{n}{k}_{q}^{\perp}$ again later.
We first prove that $E_{n}(q)$ is rank-symmetric. 

\begin{proposition}\label{rs}
The Euclidean poset $E_{n}(q)$ is rank-symmetric. i.e. $|\textup{Euc}_{k},\textup{Euc}_{n}|_{q}=|\textup{Euc}_{n-k},\textup{Euc}_{n}|_{q}$.
\end{proposition}

\begin{center}
\begin{tikzpicture}[scale=0.5]
  \node (max) at (0,4) {$\mathbb{F}_{3}^{3}$};
  \node (a) at (-2,2) {$P_{1}$};
  \node (b) at (0,2) {$P_{2}$};
  \node (c) at (2,2) {$P_{3}$};
  \node (d) at (-2,0) {$\ell_{1}$};
  \node (e) at (0,0) {$\ell_{2}$};
  \node (f) at (2,0) {$\ell_{3}$};
  \node (min) at (0,-2) {$0$};
  \draw (min) -- (d) -- (a) -- (max) -- (b) -- (f)
  (e) -- (min) -- (f) -- (c) -- (max)
  (d) -- (b);
  \draw[preaction={draw=white, -,line width=6pt}] (a) -- (e) -- (c);
\end{tikzpicture}
\captionof{figure}{The Hasse diagram of $E_{3}(3)$}\label{pic: Hasse}
\end{center}

\begin{proof}
Let $E_{1}$ be a $k$-dimensional Euclidean subspace of $\F$. Then, there is a $(n-k)$-dimensional Euclidean subspace $E_{2}$ of $\F$ such that
\[E_{1}\oplus E_{2}=\F=E_{1}\oplus E_{1}^{\perp}.\] 
By Witt's cancellation theorem, we obtain $E_{2}=E_{1}^{\perp}$. Since taking $\perp$ is bijective, the result holds. 
\end{proof}

Next, we show that $E_{n}(q)$ is rank-unimodal. 
To do this, we give a formula for $|\text{Euc}_{k},\text{Euc}_{n}|_{q}$, called the Euclidean binomial coefficient.   
We first consider a useful lemma to achieve this goal. 

\begin{lemma}\label{32}For any $1 \le k \le n$, the number of $k$-dimensional Euclidean subspaces of $(\mathbb{F}_{q}^{n},\textup{Euc}_{n})$ containing a fixed Euclidean line of $(\mathbb{F}_{q}^{n},\textup{Euc}_{n})$ is $|\textup{Euc}_{k-1},\textup{Euc}_{n-1}|_{q}$.
\end{lemma}

\begin{proof} By Witt's extension theorem, the number of $k$-dimensional Euclidean subspaces of $(\mathbb{F}_{q}^{n},\text{Euc}_{n})$ containing a Euclidean line is independent of which Euclidean line is chosen. Let $L$ be a Euclidean line. Then, we have the following bijection map:
\begin{align*}
\begin{Bmatrix}
\text{$(k-1)$-dimensional Euclidean subspaces} \\
\text{in }(\mathbb{F}_{q}^{n-1},\text{Euc}_{n-1})
\end{Bmatrix}
 ~~~&\longrightarrow~~~\begin{Bmatrix}
 k\text{-dimensional Euclidean subspaces in }\\
\F \text{ containing }L
\end{Bmatrix}\\
W \quad \quad \quad \quad \quad \quad \quad \quad \quad \quad &\mapsto \quad \quad \quad \quad \quad \quad \quad \quad \quad  L \oplus W.
\end{align*}
It follows that this map is bijective by its definition.
\end{proof}

\begin{theorem}\label{subsp} For any $1 \leq k \leq n$, we have that
\begin{equation}\label{dotb}
|\textup{Euc}_{k},\textup{Euc}_{n}|_{q}=\frac{|\textup{Euc}_{1},\textup{Euc}_{n}|_{q}|\textup{Euc}_{1},\textup{Euc}_{n-1}|_{q}\cdots|\textup{Euc}_{1},\textup{Euc}_{n-k+1}|_{q}}{|\textup{Euc}_{1},\textup{Euc}_{k}|_{q}\cdots|\textup{Euc}_{1},\textup{Euc}_{1}|_{q}}.
\end{equation}
\end{theorem}
\begin{proof}
We first claim that
\begin{equation}\label{imp}
    |\textup{Euc}_{1},\textup{Euc}_{k}|_{q}|\textup{Euc}_{k},\textup{Euc}_{n}|_{q}=|\textup{Euc}_{1},\textup{Euc}_{n}|_{q}|\textup{Euc}_{k-1},\textup{Euc}_{n-1}|_{q}.
\end{equation}
To prove the claim, we count the number of flags $0 \subset V_{1} \subset V_{2}$ in two different ways, where $V_{1}$ is a Euclidean lines of $\F$ and $V_{2}$ is a $k$-dimensional Euclidean subspaces of $\F$. Note that 
\begin{align*}
|\text{Euc}_{1},\text{Euc}_{k}|_{q}&=\text{the number of Euclidean lines of }(\mathbb{F}_{q}^{n},\text{Euc}_{k}),\\
|\text{Euc}_{k},\text{Euc}_{n}|_{q}&=\text{the number of }k\text{-dimensional Euclidean } \text{subspaces of }(\mathbb{F}_{q}^{n},\text{Euc}_{n}),\\
|\text{Euc}_{1},\text{Euc}_{n}|_{q}&=\text{the number of Euclidean lines of }(\mathbb{F}_{q}^{n},\text{Euc}_{n}),\\
|\text{Euc}_{k-1},\text{Euc}_{n-1}|_{q}&=\text{the number of }k\text{-dimensional Euclidean }\text{subspaces of }  \F  \\ & \quad \; \text{containing a fixed Euclidean line in } (\mathbb{F}_{q}^{n},\text{Euc}_{n}).
\end{align*}

The last equality is due to \cref{32}. First, we choose a $k$-dimensional Euclidean subspace of $\F$, and then find a Euclidean line inside it. One important remark is that all $k$-dimensional Euclidean subspaces of $\F$ are the same as the abstract $k$-dimensional Euclidean space by Witt's extension theorem. The second way is firstly to choose a Euclidean line of $\F$, and pick a $k$-dimensional Euclidean subspace containing it. All Euclidean lines are isometrically isomorphic by Witt's extension theorem again. Hence, we conclude the claim.

We are now ready to find the formula for $|\text{Euc}_{k},\text{Euc}_{n}|_{q}$ and we proceed by induction on $k$ and $n$. If $k=1$, \cref{dotb} holds. Suppose that \cref{dotb} holds for $1 \leq \ell_{1} \leq  k $ and $1 \leq \ell_{2} \leq n$. Then, using \cref{imp} and the induction hypothesis, we have that
\begin{align*}
|\textup{Euc}_{k+1},\textup{Euc}_{n}|_{q}&= \frac{|\text{Euc}_{1},\text{Euc}_{n}|_{q}|\text{Euc}_{k},\text{Euc}_{n-1}|_{q}}{|\text{Euc}_{1},\text{Euc}_{k+1}|_{q}}\\
&=\frac{|\textup{Euc}_{1},\textup{Euc}_{n}|_{q}|\textup{Euc}_{1},\textup{Euc}_{n-1}|_{q}\cdots|\textup{Euc}_{1},\textup{Euc}_{n-k}|_{q}}{|\textup{Euc}_{1},\textup{Euc}_{k+1}|_{q}\cdots|\textup{Euc}_{1},\textup{Euc}_{1}|_{q}}
\end{align*} 
and
\begin{align*}
|\textup{Euc}_{k},\textup{Euc}_{n+1}|_{q}&= \frac{|\text{Euc}_{1},\text{Euc}_{n+1}|_{q}|\text{Euc}_{k-1},\text{Euc}_{n}|_{q}}{|\text{Euc}_{1},\text{Euc}_{k}|_{q}}\\
&=\frac{|\textup{Euc}_{1},\textup{Euc}_{n+1}|_{q}|\textup{Euc}_{1},\textup{Euc}_{n}|_{q}\cdots|\textup{Euc}_{1},\textup{Euc}_{n-k}|_{q}}{|\textup{Euc}_{1},\textup{Euc}_{k}|_{q}\cdots|\textup{Euc}_{1},\textup{Euc}_{1}|_{q}}.
\end{align*}
This completes the proof.
\end{proof}

Furthermore, it is natural to define the following.
\begin{definition}\label{dot-binom} For any $1 \leq k\leq n$,
\renewcommand\labelitemi{\tiny$\bullet$}
\begin{itemize}
\item $[k]_{q}^{\perp}:=|\text{Euc}_{1},\text{Euc}_{k}|_{q}$,
\item $[n]_{q}^{\perp}!:=[n]_{q}^{\perp}[n-1]_{q}^{\perp}\cdots[1]_{q}^{\perp}$,
\item $\binom{n}{k}_{q}^{\perp}:=|\text{Euc}_{k},\text{Euc}_{n}|_{q}=\frac{[n]_{q}^{\perp}!}{[k]_{q}^{\perp}![n-k]_{q}^{\perp}!}$.
\end{itemize}
We call these \textbf{Euclidean-analogues}. In particular, we call $\binom{n}{k}_{q}^{\perp}$ the \textbf{Euclidean binomial coefficient}. 
We adopt the convention $[0]_{q}^{\perp}!=1$, and this implies that $|\text{Euc}_{0},\text{Euc}_{n}|_{q}=1$. 
\end{definition}

\begin{remark}
The Euclidean binomial coefficients combine the counts for the number of hyperbolic, elliptic, and parabolic subspaces of a non-degenerate space in the following way.

Suppose that $q \equiv 1$ (mod $4$), $n$ and $k$ are odd. 
As we have seen in \cref{corresp}, the parabolic form $\frac{n-1}{2}\mathbb{H} \oplus x^{2}$ is equivalent to Euc$_{n}$.
Thus, the number of $k$-dimensional parabolic subspaces of a $n$-dimensional parabolic space is $\binom{n}{k}_{q}^{\perp}$. 
Similarly, the number of $(k-1)$-dimensional hyperbolic subspaces of a $n$-dimensional parabolic space is $\binom{n}{k-1}_{q}^{\perp}$.
The number of $(k-1)$-dimensional elliptic subspaces of a $n$-dimensional parabolic space is $\binom{n}{\overline{k-1}}_{q}^{\perp}$, the number of $(k-1)$-dimensional Lorentzian subspaces of $(\mathbb{F}_{q}^{n},\textup{Euc}_{n})$, which is formulated in \cref{other limits}.

\end{remark}

\begin{lemma}\label{inc}
For any $k \geq 1$, we have $\left[ k \right ]_{q}^{\perp} \leq \left[k+1 \right ]_{q}^{\perp}$.
\end{lemma}

\begin{proof}
Note that any Euclidean line in a $k$-dimensional Euclidean space can be isometrically embedded in a $(k+1)$-dimensional Euclidean space by inclusion. Thus, the number of Euclidean lines in a $k$-dimensional Euclidean space is less than equal to the number of Euclidean lines in a $(k+1)$-dimensional Euclidean space. 
\end{proof}

\begin{proposition}\label{ru}
The Euclidean poset $E_{n}(q)$ is rank-unimodal. i.e. there is $j$ such that
\[\binom{n}{0}_{q}^{\perp} \leq \binom{n}{1}_{q}^{\perp} \leq \cdots \leq \binom{n}{j}_{q}^{\perp} \geq \cdots \geq \binom{n}{n}_{q}^{\perp}. \]
\end{proposition}

\begin{proof}
Let us consider the following:
\[M:=\frac{\binom{n}{k}_{q}^{\perp}}{\binom{n}{k-1}_{q}^{\perp}}=\frac{|\text{Euc}_{k},\text{Euc}_{n}|_{q}}{|\text{Euc}_{k-1},\text{Euc}_{n}|_{q}}=\frac{|\text{Euc}_{1},\text{Euc}_{n-k+1}|_{q}}{|\text{Euc}_{1},\text{Euc}_{k}|_{q}}.\]
Note that $M$ is greater than $1$ if $n-k+1 \geq k$ by \cref{inc}. 
This is equivalent to $\frac{n+1}{2}\geq k$. 
We also note that $M$ is less than $1$ if $n-k+1 \leq k$ by \cref{inc} again. Equivalently, $\frac{n+1}{2}\leq k$. If $n$ is odd, then we choose $j=(n+1)/2$. 
If $n$ is even, then we set $j=n/2$. It follows that $E_{n}(q)$ is rank-unimodal.
\end{proof}

There is another way to show that $E_{n}(q)$ is rank-unimodal. A sequence $a_{0},a_{1},\ldots,a_{n}$ of real numbers is called \textbf{log-concave} if $a_{k}^{2}\geq a_{k-1}a_{k+1}$ for $1 \leq k \leq n-1$, and \textbf{unimodal} if there is $j$ such that $a_{0} \leq a_{1} \leq \cdots \leq a_{j} \geq a_{j+1} \geq \cdots \geq a_{n}$. We prove that the sequence of the Euclidean binomial coefficients is log-concave.

\begin{proposition}\label{lc}
The sequence 
\[\binom{n}{0}_{q}^{\perp},\binom{n}{1}_{q}^{\perp},\ldots,\binom{n}{n}_{q}^{\perp}\]
is log-concave.
\end{proposition}

\begin{proof}
What we need to show is
\[\frac{\binom{n}{k}_{q}^{\perp}}{\binom{n}{k-1}_{q}^{\perp}} \geq \frac{\binom{n}{k+1}_{q}^{\perp}}{\binom{n}{k}_{q}^{\perp}}.\]
This inequality reduces to
\[\frac{\left [ n-k+1\right ]_{q}^{\perp}}{\left [k \right ]_{q}^{\perp}} \geq \frac{\left [ n-k \right ]_{q}^{\perp}}{\left [k+1 \right ]_{q}^{\perp}}.\]
\cref{inc} completes the proof.
\end{proof}
In general, there is a systematic method for checking log-concavity, called Newton's Theorem. This can be found in Stanley's book \cite{St1}. We introduce a useful proposition to check unimodality without proof. A sequence $a_{0},a_{1},\ldots,a_{n}$ has \textbf{no internal zeros} if $a_{i}\ne 0$ and $a_{k}\ne 0$ for $i<j<k$, then $a_{j} \ne 0$. 
\begin{proposition}\cite{St1}\label{uni}
Let $\alpha=(a_{0},a_{1},\ldots,a_{n})$ be a sequence of non-negative real numbers with no internal zeros. If $\alpha$ is log-concave, then $\alpha$ is unimodal.
\end{proposition}

Since $(\binom{n}{0}_{q}^{\perp},\binom{n}{1}_{q}^{\perp},\ldots,\binom{n}{n}_{q}^{\perp})$ is a sequence of non-negative real numbers with no internal zeros satisfying the log-concavity, $(\binom{n}{0}_{q}^{\perp},\binom{n}{1}_{q}^{\perp},\ldots,\binom{n}{n}_{q}^{\perp})$ is unimodal by \cref{uni}.

\subsection{The formula of the number of lines of Euclidean and Lorentzian types}
To count the number of Euclidean, Lorentzian, and isotropic lines of $\F$ and $(\mathbb{F}_{q}^{n},\text{Lor}_{n})$, respectively, we introduce the following Minkowski's work that counts the size of the spheres in $\mathbb{F}_{q}^{n}$, stated using the classification \cref{classi2}. 

\begin{theorem}\cite{BHIR, Ca, Mi}\label{Mi}
Let $(V,Q)$ be a non-degenerate quadratic space and $S^{Q}_{r}$ denote the sphere of radius $r$ in $(V,Q)$. $i.e.$ $S^{Q}_{r}=\left \{  x \in V ~|~Q(x)=r\right \}$. 
\begin{enumerate}
    \item[\textup{(1)}] If $Q\cong k\mathbb{H}$ and $n=2k$, then
\[|S^{Q}_{r}|=\begin{cases}
q^{n-1}-q^{\frac{n-2}{2}} & \text{ if } r\ne 0\\ 
q^{n-1}+q^{\frac{n}{2}}-q^{\frac{n-2}{2}} & \text{ if } r=0 
\end{cases}.\]
\item[\textup{(2)}] If $Q\cong (k-1)\mathbb{H}\oplus (x^{2}-\lambda y^{2})$ for some non-square $\lambda$ and $n=2k$, then
\[|S^{Q}_{r}|=\begin{cases}
q^{n-1}+q^{\frac{n-2}{2}} & \text{ if } r\ne 0 \\ 
q^{n-1}-q^{\frac{n}{2}}+q^{\frac{n-2}{2}} & \text{ if } r=0 
\end{cases}.\]
\item[\textup{(3)}] If $Q \cong k\mathbb{H} \oplus cx^{2}$ and $n=2k+1$, then
\[|S^{Q}_{r}|=\begin{cases}
q^{n-1}+q^{\frac{n- 1}{2}}\textup{sgn}(\frac{r}{c}) & \text{ if } r\ne 0 \\ 
q^{n-1} & \text{ if } r=0
\end{cases},\]
\end{enumerate}
where \textup{sgn} stands for the function on $\mathbb{F}_{q}$ given by
\[\textup{sgn}(x)=\begin{cases}
0 & \text{ if } x=0 \\ 
1 & \text{ if } x\ne 0 \text{ is a square} \\ 
-1 & \text{ otherwise}
\end{cases},\]
and $c$ is either $1$ or a non-square. Note that the cases of $r=0$ contain $0$ in their counts.
\end{theorem}
In \cite{Ca}, Casselman gives the proof of \cref{Mi} by following Minkowski's essay, 
which uses the finite Fourier transform. 
We next introduce some notations: 
\begin{enumerate}\label{notations}
    \item $\epsilon$ is $1$ if $q \equiv 1$ (mod $4$) or $n \equiv 0,1$ (mod $4$), and $-1$ otherwise, 
    \item $\delta$ is $1$ if $n$ is even, and $0$ otherwise, \item $\eta$ is $1$ if $n$ is even, and $-1$ otherwise. 
\end{enumerate}

We count the number of lines of each type in $(\mathbb{F}_{q}^{n},\text{Euc}_{n})$ and $(\mathbb{F}_{q}^{n},\text{Lor}_{n})$, respectively.

\begin{theorem}\label{dotline}
In $(\mathbb{F}_{q}^{n},\textup{Euc}_{n})$, the number of Euclidean lines is $(q^{n-1}-\eta\epsilon q^{n-1-\lfloor n/2\rfloor})/2$,
and the number of Lorentzian lines is $(q^{n-1}- \epsilon q^{n-1-\lfloor n/2\rfloor})/2$. In other words, we have

\begin{table}[H]
\renewcommand{\arraystretch}{2}
\begin{center}
\scalebox{0.87}{
\begin{tabular}{c||c|c}
\textbf{\textup{Euclidean lines}} & $\mathbf{q \equiv 1}$ \textbf{\textup{(mod $\mathbf{4}$)}} & $\mathbf{q \equiv 3}$ \textbf{\textup{(mod $\mathbf{4}$)}} \\ \hline \hline 
    $n \equiv 1 \pmod{4}$      & \multirow{2}{*}{$\frac{q^{n-1}+q^{\frac{n-1}{2}}}{2}$}               &                               $\frac{q^{n-1}+q^{\frac{n-1}{2}}}{2}$ \\ \cline{1-1} \cline{3-3} 
    $n \equiv 3 \pmod{4}$     &                               &                                $\frac{q^{n-1}-q^{\frac{n-1}{2}}}{2}$\\   \hline
    $n \equiv 0 \pmod{4}$      & \multirow{2}{*}{$\frac{q^{n-1}-q^{\frac{n-2}{2}}}{2}$}               &                              $\frac{q^{n-1}-q^{\frac{n-2}{2}}}{2}$   \\  \cline{1-1} \cline{3-3} 
    $n \equiv 2 \pmod{4}$      &                                 &                          $\frac{q^{n-1}+q^{\frac{n-2}{2}}}{2}$\\ 
\end{tabular}}
\caption{The number of Euclidean lines in $\F$}
\end{center}
\end{table}

\begin{table}[H]
\begin{center}
\renewcommand{\arraystretch}{2}
\scalebox{0.87}{
\begin{tabular}{c||c|c}
\textbf{\textup{Lorentzian lines}} & $\mathbf{q \equiv 1}$ \textbf{\textup{(mod $\mathbf{4}$)}} & $\mathbf{q \equiv 3}$ \textbf{\textup{(mod $\mathbf{4}$)}} \\ \hline \hline 
    $n \equiv 1 \pmod{4}$      & \multirow{2}{*}{$\frac{q^{n-1}-q^{\frac{n-1}{2}}}{2}$}               &                              $\frac{q^{n-1}-q^{\frac{n-1}{2}}}{2}$  \\ \cline{1-1} \cline{3-3} 
    $n \equiv 3 \pmod{4}$     &                               &                          $\frac{q^{n-1}+q^{\frac{n-1}{2}}}{2}$\\    \hline
    $n \equiv 0 \pmod{4}$      & \multirow{2}{*}{$\frac{q^{n-1}-q^{\frac{n-2}{2}}}{2}$}               &     
    
    $\frac{q^{n-1}-q^{\frac{n-2}{2}}}{2}$                  \\ \cline{1-1} \cline{3-3} 
    $n \equiv 2 \pmod{4}$      &                                 & $\frac{q^{n-1}+q^{\frac{n-2}{2}}}{2}$\\ 
\end{tabular}}
\caption{The number of Lorentzian lines in $\F$}
\end{center}
\end{table}

The number of isotropic lines in $\F$ is given in the following table.
\begin{table}[H]
\begin{center}
\renewcommand{\arraystretch}{1.5}
\scalebox{0.87}{
\begin{tabular}{c||c|c}
\textbf{\textup{Isotropic lines}} &  $\mathbf{q \equiv 1}$ \textbf{\textup{(mod $\mathbf{4}$)}}  & $\mathbf{q \equiv 3}$ \textbf{\textup{(mod $\mathbf{4}$)}}  \\ \hline \hline
$n \equiv 1 \pmod{2}$ & \multicolumn{2}{c}{$q^{n-2}+q^{n-3}+\cdots+1$}    \\ \hline
$n \equiv 0 \pmod{4}$ & \multirow{2}{*}{$(q^{n-2}+q^{n-3}+\cdots+1)+q^{\frac{n-2}{2}}$} & $(q^{n-2}+q^{n-3}+\cdots+1)+q^{\frac{n-2}{2}}$ \\ \cline{1-1} \cline{3-3}
$n \equiv 2 \pmod{4}$ &  & $(q^{n-2}+q^{n-3}+\cdots+1)-q^{\frac{n-2}{2}}$ \\ 
\end{tabular}}
\caption{The number of isotropic lines in $\F$.}
\end{center}
\end{table}

\end{theorem}

\begin{proof}[The first proof of \textup{\cref{dotline}}]
Let us count Euclidean lines first. Notice that the number of intersection points between a given Euclidean line and a sphere is always two if the radius is a square. Thus, the number of Euclidean lines in $\F$ is 
\[(\text{the size of the $(n-1)$-dimensional unit-sphere})/2\]
since we can scale the radius to $1$ if the radius is a square. Hence, our goal follows from \cref{Mi} once one matches the conditions on $\mathbb{F}_{q}^{n}$ with the correct quadratic type.

We consider the odd dimensional case where $n=2k+1$. From the relation
\[\mathbb{H}\oplus\mathbb{H}\oplus\cdots\oplus\mathbb{H}\oplus
c\text{Euc}_{1}=\text{Euc}_{2k+1},\]
it follows that $(-1)^{k}c=1$ by comparing their discriminants since $d(\mathbb{H})=-1$. Thus, we have $c=(-1)^{k}$. This implies that $\text{sgn}(r/c)=\text{sgn}(r(-1)^{k})$ by the property of sgn$(1/r)=$sgn$(r)$. Thus, sgn depends on whether $k$ is odd or even. If $q\equiv 1$ (mod $4$), then sgn$(-1)^{k}=1$ for any $k$ since $-1$ is a square. Since we are counting Euclidean lines, sgn$(r)=1$. Therefore, \cref{Mi} implies that the number of Euclidean lines is
\[\frac{q^{n-1}+q^{\frac{n- 1}{2}}\text{sgn}(r/c)}{2}=\frac{q^{n-1}+q^{\frac{n- 1}{2}}}{2}.\] 
If $q\equiv 3$ (mod $4$) and $n=4\ell+1$, let $k=2\ell$. By \cref{Mi}, the number of Euclidean lines is
\[\frac{q^{n-1}+q^{\frac{n- 1}{2}}\text{sgn}(r/c)}{2}=\frac{q^{n-1}+q^{\frac{n- 1}{2}}}{2}\] 
since sgn$(-1)^{k}=1$.
If $q\equiv 3$ (mod $4$) and $n=4\ell+3$, let $k=2\ell+1$. \cref{Mi} gives the number of Euclidean lines
\[\frac{q^{n-1}+q^{\frac{n- 1}{2}}\text{sgn}(r/c)}{2}=\frac{q^{n-1}-q^{\frac{n- 1}{2}}}{2}\] 
since sgn$(-1)^{k}=-1$.

Next, we consider the even dimensional case where $n=2k$. 
If $q \equiv 1$ (\textup{mod} $4$), or $q \equiv 3$ (\textup{mod} $4$) and $n=4\ell$, then $\textup{Euc}_{2k}$ is equivalent to the hyperbolic space $k\mathbb{H}$. 
Thus, the number of Euclidean lines is $(q^{n-1}-q^{\frac{n-2}{2}})/2$ by \cref{Mi}. 
If $q \equiv 3$ (\textup{mod} $4$) and $n=4\ell+2$, then $\textup{Euc}_{2k}$ is equivalent to $(k-1)\mathbb{H}\oplus (x^{2}-\lambda y^{2})$ for some non-square $\lambda$. 
Thus, the number of Euclidean lines is $(q^{n-1}+q^{\frac{n-2}{2}})/2$ by \cref{Mi}.

For Lorentzian lines, it is enough to replace sgn$(r)=1$ with sgn$(r)=-1$. 
For isotropic lines, since the total number of lines in $\F$ is $(q^{n}-1)/(q-1)$, one can derive their count by subtracting the number of Euclidean and Lorentzian lines from the total number of lines. 
\end{proof}

\begin{theorem}\label{lambdadot}
In $(\mathbb{F}_{q}^{n},\textup{Lor}_{n})$, the number of Euclidean lines is $(q^{n-1}+\eta\epsilon q^{n-1-\lfloor n/2\rfloor})/2$,
and the number of Lorentzian lines is $(q^{n-1}+ \epsilon q^{n-1-\lfloor n/2\rfloor})/2$. Thus, we have

\begin{table}[H]
\begin{center}
\renewcommand{\arraystretch}{2}
\scalebox{0.87}{
\begin{tabular}{c||c|c}
\textbf{\textup{Euclidean lines}} & $\mathbf{q \equiv 1}$ \textbf{\textup{(mod $\mathbf{4}$)}} & $\mathbf{q \equiv 3}$ \textbf{\textup{(mod $\mathbf{4}$)}} \\ 
\hline \hline 
    $n \equiv 1 \pmod{4}$      & \multirow{2}{*}{$\frac{q^{n-1}-q^{\frac{n-1}{2}}}{2}$}               &                              $\frac{q^{n-1}-q^{\frac{n-1}{2}}}{2}$  \\ \cline{1-1} \cline{3-3} 
    $n \equiv 3 \pmod{4}$     &                               &                                 $\frac{q^{n-1}+q^{\frac{n-1}{2}}}{2}$ \\    \hline
    $n \equiv 0 \pmod{4}$      & \multirow{2}{*}{$\frac{q^{n-1}+q^{\frac{n-2}{2}}}{2}$}               &                              $\frac{q^{n-1}+q^{\frac{n-2}{2}}}{2}$    \\ \cline{1-1} \cline{3-3} 
    $n \equiv 2 \pmod{4}$      &                                 &                                 $\frac{q^{n-1}-q^{\frac{n-2}{2}}}{2}$\\ 
\end{tabular}}
\caption{The number of Euclidean lines in $(\mathbb{F}_{q}^{n},\text{Lor}_{n})$}
\end{center}
\end{table}

\begin{table}[H]
\begin{center}
\renewcommand{\arraystretch}{2}
\scalebox{0.87}{
\begin{tabular}{c||c|c}
\textbf{\textup{Lorentzian lines}} & $\mathbf{q \equiv 1}$ \textbf{\textup{(mod $\mathbf{4}$)}} & $\mathbf{q \equiv 3}$ \textbf{\textup{(mod $\mathbf{4}$)}} \\ \hline \hline 
    $n \equiv 1 \pmod{4}$      & \multirow{2}{*}{$\frac{q^{n-1}+q^{\frac{n-1}{2}}}{2}$}               &                              $\frac{q^{n-1}+q^{\frac{n-1}{2}}}{2}$   \\ \cline{1-1} \cline{3-3} 
    $n \equiv 3 \pmod{4}$     &                               &                                 $\frac{q^{n-1}-q^{\frac{n-1}{2}}}{2}$\\    \hline
    $n \equiv 0 \pmod{4}$      & \multirow{2}{*}{$\frac{q^{n-1}+q^{\frac{n-2}{2}}}{2}$}               &                              $\frac{q^{n-1}+q^{\frac{n-2}{2}}}{2}$    \\ \cline{1-1} \cline{3-3} 
    $n \equiv 2 \pmod{4}$      &                                 &                                 $\frac{q^{n-1}-q^{\frac{n-2}{2}}}{2}$\\ 
\end{tabular}}
\caption{The number of Lorentzian lines in $(\mathbb{F}_{q}^{n},\text{Lor}_{n})$}
\end{center}
\end{table}

The number of isotropic lines in $(\mathbb{F}_{q}^{n},\textup{Lor}_{n})$ is given in the following table.
\begin{table}[H]
\begin{center}
\renewcommand{\arraystretch}{1.5}
\scalebox{0.9}{
\begin{tabular}{c||c|c}
\textbf{\textup{Isotropic lines}} &  $\mathbf{q \equiv 1}$ \textbf{\textup{(mod $\mathbf{4}$)}} & $\mathbf{q \equiv 3}$ \textbf{\textup{(mod $\mathbf{4}$)}}  \\ \hline \hline
$n \equiv 1 \pmod{2}$ & \multicolumn{2}{c}{$q^{n-2}+q^{n-3}+\cdots+1$}    \\ \hline
$n \equiv 0 \pmod{4}$ & \multirow{2}{*}{$(q^{n-2}+q^{n-3}+\cdots+1)-q^{\frac{n-2}{2}}$} & $(q^{n-2}+q^{n-3}+\cdots+1)-q^{\frac{n-2}{2}}$ \\ \cline{1-1} \cline{3-3}
$n \equiv 2 \pmod{4}$ &  & $(q^{n-2}+q^{n-3}+\cdots+1)+q^{\frac{n-2}{2}}$ \\ 
\end{tabular}}
\caption{The number of isotropic lines in $(\mathbb{F}_{q}^{n},\text{Lor}_{n})$}
\end{center}
\end{table}

\end{theorem}

\begin{proof}
We follow the same strategy as in \cref{dotline}. Let us first consider the odd dimensional case where $n=2k+1$ and we count Euclidean line of $(\mathbb{F}_{q}^{n},\text{Lor}_{2k+1})$. From the relation
\[\mathbb{H}\oplus\mathbb{H}\oplus\cdots\oplus\mathbb{H}\oplus c\text{Euc}_{1}=\text{Lor}_{2k+1},\]
by comparing discriminants, we obtain $(-1)^{k}c=\lambda$ for some non-square $\lambda$. i.e. $c=(-1)^{k}\lambda$. Thus, we have
\[\text{sgn}(c)=
\begin{cases}
-1 & \text{ if } k \text{ is even or }q\equiv 1\text{ (mod $4$)} \\ 
1 & \text{ if } k \text{ is odd and }q \equiv 3 \text{ (mod $4$)}
\end{cases}
.\]
Since sgn$(r/c)=$sgn$(rc)$ and sgn$(r)=1$,  
\[\text{sgn}(r/c)=
\begin{cases}
-1 & \text{ if } k \text{ is even or }q\equiv 1\text{ (mod $4$)} \\ 
1 & \text{ if } k \text{ is odd and }q \equiv 3 \text{ (mod $4$)}
\end{cases}
.\]
This observation and similar counts as the proof of \cref{dotline} complete the proof. 
\end{proof}
We give another proof of \cref{dotline}. 
One can prove \cref{lambdadot} using a similar way. 
To do this, we need one more lemma, which counts how many pairs $(n,n+1)$ in $\mathbb{F}_{q}$ are $(S,S),(S,N),(N,S)$, and $(N,N)$, where $S$ denotes a square and $N$ denotes a non-square in $\mathbb{F}_{q}$. For example, in $\mathbb{F}_{5}$, there is no $(S,S)$ pair, one $(S,N)$ pair $(1,2)$, one $(N,S)$ pair $(3,4)$ and one $(N,N)$ pair $(2,3)$. 
In \cite{Da}, a proof of the following lemma is given in the case that $q$ is a prime.
Here we consider a general prime power $q$, and our proof is similar.
We denote by $|SS|,|SN|,|NS|$, and $|NN|$ the number of $(S,S),(S,N),(N,S)$, and $(N,N)$ pairs in $\mathbb{F}_{q}$, respectively, and we do not consider the pairs $(0,1)$ and $(-1,0)$.
\begin{lemma}\cite{Da}\label{ss} Assume that $n \in \mathbb{F}_{q} \setminus \{0,-1\}$. The number of pairs $(n,n+1)$ among squares and non-squares for each case is given in the following table. 
\begin{table}[H]
\begin{center}
\renewcommand{\arraystretch}{1.1}
\scalebox{0.9}{
\begin{tabular}{c||c|c}
 & $\mathbf{q \equiv 1}$ \textbf{(\textup{mod} $\mathbf{4}$)} & $\mathbf{q\equiv 3}$ \textbf{(\textup{mod} $\mathbf{4}$)} \\ \hline \hline 
$|SS|$ & $(q-5)/4$ & $(q-3)/4$ \\ \hline
$|SN|$ & $(q-1)/4$ & $(q+1)/4$ \\ \hline
$|NS|$ & $(q-1)/4$ & $(q-3)/4$\\ \hline
$|NN|$ & $(q-1)/4$ & $(q-3)/4$ \\ 
\end{tabular}}
\caption{The number of pairs for each case in $\mathbb{F}_{q}$}
\end{center}
\end{table}
\end{lemma}
\begin{proof}
We first notice the following four equations:
\begin{align*}
|SS|+|SN|&=\begin{cases}
(q-3)/2 & \text{ if } q \equiv 1\text{ (mod $4$)}\\ 
(q-1)/2 & \text{ if } q \equiv 3\text{ (mod $4$)}
\end{cases},  &|SS|+|NS|=(q-3)/2,\\
|NS|+|NN|&=\begin{cases}
(q-1)/2 & \text{ if } q \equiv 1\text{ (mod $4$)}\\ 
(q-3)/2 & \text{ if } q \equiv 3\text{ (mod $4$)}
\end{cases},  &|SN|+|NN|=(q-1)/2.
\end{align*}
We find some relations among $|SS|,|SN|,|NS|$, and $|NN|$ to solve simultaneous equations. From the following equation 
\[\left (\frac{n}{q} \right)\left(\frac{n+1}{q}\right)=\begin{cases}
1 & \text{ if } (n,n+1)=SS,NN \\ 
-1 & \text{ if } (n,n+1)=SN,NS
\end{cases},\]
where $(k/q)$ is a Jacobi symbol, we obtain that
\[|SS|+|NN|-|SN|-|NS|=\sum_{n \in \mathbb{F}_{q} \setminus \left \{ -1,0 \right \}} \left ( \frac{n(n+1)}{q}\right).\]
Since any $n$ has a reciprocal in $\mathbb{F}_{q}^{*}$, denoted by $m$, it follows that $n(n+1) = n^{2}(1+m)$ in $\mathbb{F}_{q}^{*}$. Note that the Jacobi symbol is multiplicative. Thus, we simplify the terms in the summation: 
\[\left ( \frac{n(n+1)}{q}\right)=\left ( \frac{1+m}{q}\right).\]
Since $n \in \mathbb{F}_{q}\setminus \left \{ -1,0 \right \}$, we have $m \in \mathbb{F}_{q}\setminus \left \{ -1,0 \right \}$, which implies
\[\sum_{n \in \mathbb{F}_{q} \setminus \left \{ -1,0 \right \}} \left ( \frac{n(n+1)}{q}\right)=\sum_{m \in \mathbb{F}_{q} \setminus \left \{ -1,0 \right \}} \left ( \frac{1+m}{q}\right).\]
We note that the numbers of squares and non-squares are the same.
This yields the equation $\sum_{m \in \mathbb{F}_{q} } \left (m/q \right )=0.$
Putting it altogether, we have that
\[
|SS|+|NN|-|SN|-|NS|=\sum_{m \in \mathbb{F}_{q} \setminus \left \{ -1,0 \right \}} \left ( \frac{1+m}{q}\right)=-1+\sum_{m \in \mathbb{F}_{q} } \left (\frac{m}{q} \right )=-1.
\]
By some computations, $|SS|,|NN|,|SN|$, and $|NS|$ can be obtained.
\end{proof}

\begin{proof}[The second proof of \textup{\cref{dotline}}]
We prove the theorem by mathematical induction, and consider the case where $q \equiv 1$ (mod $4$) and $n$ is odd. 
The proof is similar for other cases. 

Let $n=1$. In $(\mathbb{F}_{q},\textup{Euc}_{1})$, there is only one Euclidean line, which is the same as saying $(q^{1-1}+q^{\frac{1-1}{2}})/2=1$. On the other hand, there are no Lorentzian lines, which is also the same with $(q^{1-1}-q^{\frac{1-1}{2}})/2=0$. For isotropic lines, there is one isotropic line $\left \langle 0  \right \rangle$, implying $q^{1-1}=1$. 

Suppose that this statement is valid for $k \leq n$. 
We count the number of lines of Euclidean lines in $(\mathbb{F}_{q}^{n+1},\textup{Euc}_{n+1})$. 
The proof is similar for counting Lorentzian lines. 
Let $e_{1},e_{2},\ldots,e_{n+1}$ be the standard unit vectors of $(\mathbb{F}_{q}^{n+1},\textup{Euc}_{n+1})$.
Then, $e_{n+1}+t_{1}e_{1}+t_{2}e_{2}+\cdots +t_{n}e_{n}$ can consist of any line in $\mathbb{F}_{q}^{n+1} \setminus \left \langle e_{1},e_{2},\ldots,e_{n} \right \rangle $, except for $\left< e_{n+1} \right>$, where $t_{i}$ in $\mathbb{F}_{q}$ for $i=1,2,\ldots, n$. 
In order to find a condition when a line in $\mathbb{F}_{q}^{n+1} \setminus \left \langle e_{1},e_{2},\ldots,e_{n}  \right \rangle $ is Euclidean, we observe the following:
\begin{equation}\label{dot}
( e_{n+1}+t_{1}e_{1}+\cdots +t_{n}e_{n})\cdot (e_{n+1}+t_{1}e_{1}+\cdots +t_{n}e_{n} ) =1+t_{1}^{2}+\cdots+t_{n}^{2}.
\end{equation}
In \cref{dot}, when the right-hand side is a square, then the number of solutions is the difference between the number of Euclidean lines in $(\mathbb{F}_{q}^{n+1},\textup{Euc}_{n+1})$ and the number of Euclidean lines in $(\mathbb{F}_{q}^{n},\textup{Euc}_{n})$. 
Suppose that $q \equiv 1$ (mod $4$). Let $a_{n}$ be the sequence of the number of Euclidean lines in $(\mathbb{F}_{q}^{n},\textup{Euc}_{n})$. Note that $\left |SS  \right |=(q-5)/4$ and $\left | NS \right |=(q-1)/4$ by \cref{ss}. 
If $k$ is odd, 
\begin{align*}
a_{k+1}-a_{k}&=|S_{0}^{k-1}|+|\sum_{c^{2}\ne 0}t_{1}^{2}+t_{2}^{2}+\cdots+t_{k}^{2}=c^{2}-1|\\
&=|S_{0}^{k-1}| +\Big |\sum_{\substack{c^{2}\ne 0, \\ c^{2}-1:\text{ square}}} t_{1}^{2}  +\cdots+t_{k}^{2}=c^{2}-1 \Big|  +  \Big|\sum_{\substack{c^{2}\ne 0, \\ c^{2}-1:\text{ non-square}}}t_{1}^{2}+\cdots+t_{k}^{2}=c^{2}-1 \Big|\\
&=q^{k-1}+\frac{1}{4}(q-5)(q^{k-1}+q^{\frac{k-1}{2}})+\frac{1}{4}(q-1)(q^{k-1}-q^{\frac{k-1}{2}})\\
&=\frac{q-1}{2}q^{k-1}-q^{\frac{k-1}{2}},
\end{align*}
where $|S_{0}^{k-1}|$ is the number of size of the $(k-1)$-sphere with the radius $0$. Similarly, when $k$ is even, we obtain
\[a_{k+1}-a_{k}=\frac{q-1}{2}q^{k-1}-\frac{q-1}{2}q^{\frac{k-2}{2}}+q^{\frac{k}{2}}.\]
We are now ready to find $a_{n+1}$ when $n$ is odd:
\begin{align*}
a_{n+1}-a_{1}&=\sum_{k=1}^{n}(a_{k+1}-a_{k})\\
&=\sum_{\substack{k=1, \\ k:\text{ odd}}}^{n}(a_{k+1}-a_{k})+\sum_{\substack{k=1, \\ k:\text{ even}}}^{n}(a_{k+1}-a_{k})\\
&=\sum_{\substack{k=1, \\ k:\text{ odd}}}^{n}\left(\frac{q-1}{2}q^{k-1}-q^{\frac{k-1}{2}}\right)+\sum_{\substack{k=1, \\ k:\text{ even}}}^{n}\left(\frac{q-1}{2}q^{k-1}-\frac{q-1}{2}q^{\frac{k-2}{2}}+q^{\frac{k}{2}}\right)\\
&=\frac{q-1}{2}\frac{(q^{2})^{\frac{n+1}{2}}-1}{q^{2}-1}-\frac{q^{\frac{n+1}{2}}-1}{q-1}+q\frac{q-1}{2}\frac{(q^{2})^{\frac{n-1}{2}}-1}{q^{2}-1}-\frac{q-1}{2}\frac{q^{\frac{n-1}{2}}-1}{q-1}+q\frac{q^{\frac{n-1}{2}}-1}{q-1}\\
&=\frac{q^{n}-q^{\frac{n-1}{2}}}{2}-1.
\end{align*}
Since $a_{1}=1$, we conclude that $a_{n+1}=(q^{n}-q^{\frac{n-1}{2}})/2$. This completes the proof.
\end{proof}

\subsection{An isometric invariant of quadratic spaces over finite fields}\label{iso inv}

In this subsection, we show that the number of lines of each type is an intrinsic invariant of quadratic spaces over finite fields.

Since there are two types of non-degenerate quadratic spaces over finite fields due to \cref{classi1}, 
we have the following $2k+1$ possible $k$-dimensional quadratic subspaces 
of $(\mathbb{F}_{q}^{n},\textup{Euc}_{n})$ up to equivalence:
\[
\textup{Euc}_{k}, \quad \textup{Lor}_{k}, \quad \textup{Euc}_{k-1}\oplus 0,\quad \textup{Lor}_{k-1}\oplus 0,\quad \ldots \quad  ,\quad , \textup{Euc}_{1}\oplus 0^{k-1},\quad \textup{Lor}_{1}\oplus 0^{k-1}, \quad 0^{k}.
\]
It is not true in general that all these types of quadratic subspaces exist inside of $(\mathbb{F}_{q}^{n},\text{Euc}_{n})$. 
It is well-known which types of quadratic subspaces 
can be embedded in $(\mathbb{F}_{q}^{n},\text{Euc}_{n})$. (For a reference, see Section 2.5 in \cite{Yo4}.) 
Let us call a $k$-dimensional degenerate quadratic subspace $S$ of $(\mathbb{F}_{q}^{n},\text{Euc}_{n})$ a $\mathbf{A}$\textbf{-(sub)space} if $S$ is isometrically isomorphic to a type A of quadratic subspaces as listed above. 
For example, if a quadratic subspace $S$ of $(\mathbb{F}_{q}^{n},\text{Euc}_{n})$ is isometrically isomorphic to Euc$_{k-1}\oplus 0$, then $S$ is called a Euc$_{k-1}\oplus 0$-subspace of $(\mathbb{F}_{q}^{n},\text{Euc}_{n})$.

Recall that dimension and discriminant are useful and effective invariants of non-degenerate quadratic spaces over finite fields (\cref{coro}). 
However, they are not good enough to differentiate all types of quadratic spaces above including degenerate spaces. 
We introduce a new isometric invariant of combinatorial type on $(\mathbb{F}_{q}^{n},\textup{Euc}_{n})$. Our new invariant, the number of lines of certain types, works well on $(\mathbb{F}_{q}^{n},\textup{Euc}_{n})$ especially when we want to distinguish even degenerate quadratic spaces as well as non-degenerate quadratic spaces.

\begin{theorem}\label{invariant}
Let $W_{1}$ and $W_{2}$ be quadratic spaces over a finite field. Then, $W_{1}$ and $W_{2}$ are isometrically isomorphic if and only if $(e^{(1)},\ell^{(1)},i^{(1)})=(e^{(2)},\ell^{(2)},i^{(2)})$, where $e^{(j)}$ is the number of Euclidean lines in $W_{j}$, $\ell^{(j)}$ is the number of Lorentzian lines in $W_{j}$, and $i^{(j)}$ is the number of isotropic lines in $W_{j}$ for $j=1,2$. 
\end{theorem}

\begin{proof}
If two quadratic subspaces are isometrically isomorphic, then the number of lines of each type in them is the same since isometries do not change the line types. 
To prove the converse, let us denote the number of Euclidean lines in $\F$ by $s_{n}$, and Lorentzian lines in $\F$ by $t_{n}$. 
For a vector $v$ in a $\text{Euc}_{k}\oplus 0^{n-k}$-space, we can write $v=(v_{1},v_{2})$ with $|v|=|v_{1}|$, where $v_{1}$ is a nonzero vector in the $k$-dimensional Euclidean space and $v_{2}$ is a vector in the $0^{n-k}$-space. 
Thus, the number of Euclidean lines in a Euc$_{k} \oplus 0^{n-k}$-space is $s_{k} \times q^{n-k}$. 
Similarly, the number of Lorentzian lines in a Euc$_{k} \oplus 0^{n-k}$-space is $t_{k} \times q^{n-k}$.
For the number of isotropic lines in a Euc$_{k} \oplus 0^{n-k}$-space, we subtract the number of Euclidean and Lorentzian lines from the total number of lines in $\mathbb{F}_{q}^{n}$. 
Let us consider the case when $q \equiv 1$ (mod $4$) and $n$ is odd. 
The number of lines of each type, in this case, is given in \cref{numbers}. 
One can check that the triples in each quadratic type are all distinct. 
The remaining cases can be obtained similarly.
\end{proof}

\begin{table}[H]
\renewcommand{\arraystretch}{1.5}
\scalebox{0.9}{
\begin{tabular}{c||c|c|c}
\textbf{The type of spaces} & \textbf{Euclidean lines} & \textbf{Lorentzian lines}  & \textbf{Isotropic lines} \\ \hline \hline 
Euc$_{n}$ & $(q^{n-1}+q^{\frac{n-1}{2}})/2$ & $(q^{n-1}-q^{\frac{n-1}{2}})/2$ & $q^{n-2}+q^{n-3}+\cdots+q+1$ \\ \hline
Lor$_{n}$ & $(q^{n-1}-q^{\frac{n-1}{2}})/2$ & $(q^{n-1}+q^{\frac{n-1}{2}})/2$ & $q^{n-2}+q^{n-3}+\cdots+q+1$ \\ \hline
Euc$_{k}\oplus 0^{n-k}$, $k$ is odd & $(q^{n-1}+q^{\frac{2n-k-1}{2}})/2$ & $(q^{n-1}-q^{\frac{2n-k-1}{2}})/2$ & $q^{n-2}+q^{n-3}+\cdots+q+1$ \\ \hline
Euc$_{k}\oplus 0^{n-k}$, $k$ is even & $(q^{n-1}-q^{\frac{2n-k-2}{2}})/2$ & $(q^{n-1}-q^{\frac{2n-k-2}{2}})/2$ & $q^{n-2}+q^{n-3}+\cdots+q+1+q^{\frac{2n-k-2}{2}}$  \\ \hline
Lor$_{k}\oplus 0^{n-k}$, $k$ is odd & $(q^{n-1}-q^{\frac{2n-k-1}{2}})/2$ & $(q^{n-1}+q^{\frac{2n-k-1}{2}})/2$ & $q^{n-2}+q^{n-3}+\cdots+q+1$
\\ \hline 
Lor$_{k}\oplus 0^{n-k}$, $k$ is even & $(q^{n-1}+q^{\frac{2n-k-2}{2}})/2$ & $(q^{n-1}+q^{\frac{2n-k-2}{2}})/2$ & $q^{n-2}+q^{n-3}+\cdots+q+1-q^{\frac{2n-k-2}{2}}$ \\ \hline
$0^{n}$ & $0$ & $0$ & $q^{n-1}+q^{n-2}+\cdots+q+1$ \\ 
\end{tabular}}
\caption{The number of Euclidean, Lorentzian, and isotropic lines in each quadratic space when $q \equiv 1$ (mod $4$) and $n$ is odd}
\label{numbers}
\end{table}

\begin{example}\label{isometric invariant} 
We revisit \cref{all2}. One can check that \cref{invariant} works for $2$-dimensional quadratic subspaces of $(\mathbb{F}_{3}^{3},\textup{Euc}_{3})$. Here, $E$ denotes Euclidean line, $L$ denotes Lorentzian line, and $I$ denotes isotropic line. 

\begin{table}[H]
\begin{center}
\renewcommand{\arraystretch}{1.2}
\scalebox{0.9}{
\begin{tabular}{c||c|c|c}
 & \textbf{Lines in each plane} & \textbf{Type of each line} & \textbf{Type of planes} \\ \hline \hline
$p_{1}$ & $\left \{ \left \langle (1,0,0) \right \rangle,\left \langle (0,1,0) \right \rangle, \left \langle (1,1,0) \right \rangle, \left \langle (1,2,0) \right \rangle \right \}$  &
\multirow{3}{*}{$E,E,L,L$} & \multirow{3}{*}{$\text{Euclidean}$} \\ \cline{1-2}
$p_{2}$ & $\left \{ \left \langle (1,0,0) \right \rangle,\left \langle (0,0,1) \right \rangle, \left \langle (1,0,1) \right \rangle, \left \langle (1,0,2) \right \rangle \right \}$  &
 &  \\ \cline{1-2}
$p_{3}$ & $\left \{ \left \langle (0,1,0) \right \rangle,\left \langle (0,0,1) \right \rangle, \left \langle (0,1,1) \right \rangle, \left \langle (0,1,2) \right \rangle \right \}$  &
 &  \\ \hline
$p_{4}$ & $\left \{ \left \langle (1,0,0) \right \rangle,\left \langle (0,1,1) \right \rangle, \left \langle (1,1,1) \right \rangle, \left \langle (1,2,2) \right \rangle \right \}$  &
\multirow{6}{*}{$E,L,I,I$} & 
\multirow{6}{*}{$\text{Lorentzian}$} \\ \cline{1-2}
$p_{5}$ & $\left \{ \left \langle (1,0,0) \right \rangle,\left \langle (0,1,2) \right \rangle, \left \langle (1,1,2) \right \rangle, \left \langle (1,2,1) \right \rangle \right \}$  &
 &  \\ \cline{1-2}
$p_{6}$ & $\left \{ \left \langle (0,1,0) \right \rangle,\left \langle (1,0,1) \right \rangle, \left \langle (1,1,1) \right \rangle, \left \langle (1,2,1) \right \rangle \right \}$  &
 &  \\ \cline{1-2}
$p_{7}$ & $\left \{ \left \langle (0,1,0) \right \rangle,\left \langle (1,0,2) \right \rangle, \left \langle (1,1,2) \right \rangle, \left \langle (1,2,2) \right \rangle \right \}$  &
 &  \\ \cline{1-2}
$p_{8}$ & $\left \{ \left \langle (0,0,1) \right \rangle,\left \langle (1,1,0) \right \rangle, \left \langle (1,1,1) \right \rangle, \left \langle (1,1,2) \right \rangle \right \}$  &
 &  \\ \cline{1-2}
$p_{9}$ & $\left \{ \left \langle (0,0,1) \right \rangle,\left \langle (1,2,0) \right \rangle, \left \langle (1,2,1) \right \rangle, \left \langle (1,2,2) \right \rangle \right \}$  &
 &  \\ \hline
$p_{10}$ & $\left \{ \left \langle (1,0,1) \right \rangle,\left \langle (0,1,1) \right \rangle, \left \langle (1,1,2) \right \rangle, \left \langle (1,2,0) \right \rangle \right \}$  &
\multirow{4}{*}{$L,L,L,I$} & \multirow{4}{*}{$\text{Lor}\oplus 0$} \\ \cline{1-2}
$p_{11}$ & $\left \{ \left \langle (1,0,1) \right \rangle,\left \langle (1,1,0) \right \rangle, \left \langle (1,2,2) \right \rangle, \left \langle (0,1,2) \right \rangle \right \}$  &
 &  \\ \cline{1-2}
$p_{12}$ & $\left \{ \left \langle (1,1,0) \right \rangle,\left \langle (0,1,1) \right \rangle, \left \langle (1,2,1) \right \rangle, \left \langle (1,0,2) \right \rangle \right \}$  &
 & \\ \cline{1-2}
$p_{13}$ & $\left \{ \left \langle (1,1,1) \right \rangle,\left \langle (1,2,0) \right \rangle, \left \langle (1,0,2) \right \rangle, \left \langle (0,1,2) \right \rangle \right \}$  &
 &  \\ 
\end{tabular}}
\caption{The description of all $2$-dimensional quadratic subspaces of $(\mathbb{F}_{3}^{3},\textup{Euc}_{3})$ with their quadratic types and the types of lines}
\end{center}
\end{table}

\end{example}

\subsection{Combinatorial properties of $\binom{n}{k}_{q}^{\perp}$}

Here are the analogues of Pascal identities.
\begin{proposition}\label{pa} For any $1 \le  k < n$, we have
\begin{enumerate}

\item[\textup{(1)}] $\binom{n}{0}_{q}^{\perp}=\binom{n}{n}_{q}^{\perp}=1$,

\item[\textup{(2)}] $\binom{n}{k}_{q}^{\perp}=\binom{n}{n-k}_{q}^{\perp}$,

\item[\textup{(3)}] \[\binom{n}{k}_{q}^{\perp}=\binom{n-1}{k-1}_{q}^{\perp}+\frac{\left [ n \right ]_{q}^{\perp}-\left [ k \right ]_{q}^{\perp}}{\left [ n-k \right ]_{q}^{\perp}}\binom{n-1}{k}_{q}^{\perp}.\]
\end{enumerate} 
\end{proposition}

\begin{proof}
By \cref{rs},  $\binom{n}{k}_{q}^{\perp}=\binom{n}{n-k}_{q}^{\perp}$ holds. The third equality is given by 
\[
\binom{n}{k}_{q}^{\perp}-\binom{n-1}{k-1}_{q}^{\perp}=\frac{\left [ n \right ]_{q}^{\perp}-\left [ k \right ]_{q}^{\perp}}{\left [ k \right ]_{q}^{\perp}}\binom{n-1}{k-1}_{q}^{\perp}
=\frac{\left [ n \right ]_{q}^{\perp}-\left [ k \right ]_{q}^{\perp}}{\left [ n-k \right ]_{q}^{\perp}}\binom{n-1}{k}_{q}^{\perp}.
\]
\end{proof}

\begin{corollary}\label{pa2}
For any $1 \le k < n$, we have
\[\binom{n}{k}_{q}^{\perp}=\binom{n-1}{k}_{q}^{\perp}+\frac{\left [ n \right ]_{q}^{\perp}-\left [ n-k \right ]_{q}^{\perp}}{\left [ k \right ]_{q}^{\perp}}\binom{n-1}{k-1}_{q}^{\perp}.\]
\end{corollary}

One can define the analogue of Pascal's triangle by \cref{pa} or \cref{pa2}. 
\begin{example}
Let $q \equiv 1$ (mod $4$). Then, we have an analogue of Pascal's triangle with rows $0$ through $4$
\begin{center}
\begin{tabular}{>{$n=}l<{$\hspace{12pt}}*{13}{c}}
$0$ &&&&&&&$1$&&&&&&\\
$1$ &&&&&&$1$&&$1$&&&&&\\
$2$ &&&&&1&&$\frac{q-1}{2}$&&1&&&&\\
$3$ &&&&$1$&&$\frac{q^{2}+q}{2}$&&$\frac{q^{2}+q}{2}$&&$1$&&&\\
$4$ &&&$1\quad$&&$\frac{q^{3}-q}{2}$&&$\frac{q^{2}(q+1)^{2}}{2}$&&$\frac{q^{3}-q}{2}$&&$\quad1$&&
\end{tabular}
\end{center}
and 
\[\binom{4}{2}_{q}^{\perp}=\binom{3}{1}_{q}^{\perp}+\frac{\left [ 4 \right ]_{q}^{\perp}-\left[ 2 \right ]_{q}^{\perp}}{\left [ 2 \right ]_{q}^{\perp}}\binom{3}{2}_{q}^{\perp}=\frac{q^{2}+q}{2}+\frac{q^{3}-2q+1}{q-1}\frac{q^{2}+q}{2}=\frac{q^{2}(q+1)^{2}}{2}.\]
\end{example}

Next, using \cref{dotline}, we show that the Euclidean binomial coefficients are written in terms of the $q$-binomial coefficients and some polynomials.

\begin{proposition}\label{polynomials}
The expression of the Euclidean binomial coefficient $\binom{n}{k}_{q}^{\perp}$ is given by the $q$-binomial coefficients as in \textup{\cref{Table: Euc1}} and \textup{\cref{Table: Euc2}}. 
\end{proposition}

\vspace*{-\baselineskip}

\begin{table}[H]
\begin{center}
\renewcommand{\arraystretch}{1.7}
\scalebox{0.9}{
\begin{tabular}{c||c|c}
$\mathbf{\binom{n}{k}_{q}^{\perp}}$ &  $\mathbf{k}$ \textbf{\textup{is odd}} & $\mathbf{k}$ \textbf{\textup{is even}} \\ \hline \hline
$\mathbf{n}$ \textbf{\textup{is odd}} &  $\frac{1}{2}q^{\frac{k(n-k)}{2}}(q^{\frac{n-k}{2}}+1)\binom{(n-1)/2}{(k-1)/2}_{q^{2}}$ & 
$\frac{1}{2}q^{\frac{k(n-k)}{2}}(q^{\frac{k}{2}}+1)\binom{(n-1)/2}{k/2}_{q^{2}}$ \\ \hline
$\mathbf{n}$ \textbf{\textup{is even}} &  $\frac{1}{2}q^{\frac{k(n-k)-1}{2}}(q^{\frac{n}{2}}-1)\binom{(n-2)/2}{(k-1)/2}_{q^{2}}$ & 
$\frac{1}{2}q^{\frac{k(n-k)}{2}}\frac{(q^{\frac{n-k}{2}}+1)(q^{\frac{k}{2}}+1)}{q^{\frac{n}{2}}+1}\binom{n/2}{k/2}_{q^{2}}$ \\ 
\end{tabular}}.
\caption{The expression of $\binom{n}{k}_{q}^{\perp}$ when $q \equiv 1$ (mod $4$)}\label{Table: Euc1}
\end{center}
\end{table}

\begin{table}[H]
\begin{center}
\renewcommand{\arraystretch}{1.7}
\scalebox{0.9}{
\begin{tabular}{c||c|c}
$\mathbf{\binom{n}{k}_{q}^{\perp}}$ &  $\mathbf{k \equiv 1 \pmod{4}}$ & $\mathbf{k \equiv 2 \pmod{4}}$  \\ \hline \hline
$\mathbf{n \equiv 1 \pmod{4}}$ &  $\frac{1}{2}q^{\frac{k(n-k)}{2}}(q^{\frac{n-k}{2}}+1)\binom{(n-1)/2}{(k-1)/2}_{q^{2}}$ & 
$\frac{1}{2}q^{\frac{k(n-k)}{2}}(q^{\frac{k}{2}}-1)\binom{(n-1)/2}{k/2}_{q^{2}}$  \\ \hline
$\mathbf{n \equiv 2 \pmod{4}}$ &  $\frac{1}{2}q^{\frac{k(n-k)-1}{2}}(q^{\frac{n}{2}}+1)\binom{(n-2)/2}{(k-1)/2}_{q^{2}}$ & 
$\frac{1}{2}q^{\frac{k(n-k)}{2}}\frac{(q^{\frac{n-k}{2}}+1)(q^{\frac{k}{2}}-1)}{q^{\frac{n}{2}}-1}\binom{n/2}{k/2}_{q^{2}}$  \\  \hline
$\mathbf{n \equiv 3 \pmod{4}}$ &  $\frac{1}{2}q^{\frac{k(n-k)}{2}}(q^{\frac{n-k}{2}}-1)\binom{(n-1)/2}{(k-1)/2}_{q^{2}}$ & 
$\frac{1}{2}q^{\frac{k(n-k)}{2}}(q^{\frac{k}{2}}-1)\binom{(n-1)/2}{k/2}_{q^{2}}$ \\ \hline
$\mathbf{n \equiv 0 \pmod{4}}$ &  $\frac{1}{2}q^{\frac{k(n-k)-1}{2}}(q^{\frac{n}{2}}-1)\binom{(n-2)/2}{(k-1)/2}_{q^{2}}$ & 
$\frac{1}{2}q^{\frac{k(n-k)}{2}}\frac{(q^{\frac{n-k}{2}}-1)(q^{\frac{k}{2}}-1)}{q^{\frac{n}{2}}+1}\binom{n/2}{k/2}_{q^{2}}$ \\ 
\end{tabular}}

\bigskip
\bigskip

\scalebox{0.9}{
\begin{tabular}{c||c|c}
$\mathbf{\binom{n}{k}_{q}^{\perp}}$ &  $\mathbf{k \equiv 3 \pmod{4}}$ & $\mathbf{k \equiv 0 \pmod{4}}$  \\ \hline \hline
$\mathbf{n \equiv 1 \pmod{4}}$ &  $\frac{1}{2}q^{\frac{k(n-k)}{2}}(q^{\frac{n-k}{2}}-1)\binom{(n-1)/2}{(k-1)/2}_{q^{2}}$ & 
$\frac{1}{2}q^{\frac{k(n-k)}{2}}(q^{\frac{k}{2}}+1)\binom{(n-1)/2}{k/2}_{q^{2}}$  \\ \hline
$\mathbf{n \equiv 2 \pmod{4}}$ &  $\frac{1}{2}q^{\frac{k(n-k)-1}{2}}(q^{\frac{n}{2}}+1)\binom{(n-2)/2}{(k-1)/2}_{q^{2}}$ & 
$\frac{1}{2}q^{\frac{k(n-k)}{2}}\frac{(q^{\frac{n-k}{2}}-1)(q^{\frac{k}{2}}+1)}{q^{\frac{n}{2}}-1}\binom{n/2}{k/2}_{q^{2}}$  \\  \hline
$\mathbf{n \equiv 3 \pmod{4}}$ &  $\frac{1}{2}q^{\frac{k(n-k)}{2}}(q^{\frac{n-k}{2}}+1)\binom{(n-1)/2}{(k-1)/2}_{q^{2}}$ & 
$\frac{1}{2}q^{\frac{k(n-k)}{2}}(q^{\frac{k}{2}}+1)\binom{(n-1)/2}{k/2}_{q^{2}}$ \\ \hline
$\mathbf{n \equiv 0 \pmod{4}}$ &  $\frac{1}{2}q^{\frac{k(n-k)-1}{2}}(q^{\frac{n}{2}}-1)\binom{(n-2)/2}{(k-1)/2}_{q^{2}}$ & 
$\frac{1}{2}q^{\frac{k(n-k)}{2}}\frac{(q^{\frac{n-k}{2}}+1)(q^{\frac{k}{2}}+1)}{q^{\frac{n}{2}}+1}\binom{n/2}{k/2}_{q^{2}}$ \\
\end{tabular}}.
\caption{The expression of $\binom{n}{k}_{q}^{\perp}$ when $q \equiv 3$ (mod $4$)}\label{Table: Euc2}
\end{center}
\end{table}

\begin{proof}
Suppose that $q \equiv 1$ (mod $4$). If $n$ and $k$ are odd, then we have 
\begin{align*}
\binom{n}{k}_{q}^{\perp}&=\frac{|\text{Euc}_{1},\text{Euc}_{n}|_{q}|\text{Euc}_{1},\text{Euc}_{n-1}|_{q}\cdots|\text{Euc}_{1},\text{Euc}_{n-k+1}|_{q}}{|\text{Euc}_{1},\text{Euc}_{k}|_{q}\cdots|\text{Euc}_{1},\text{Euc}_{1}|_{q}}\\
&=\frac{1}{2}\frac{(q^{n-1}+q^{\frac{n-1}{2}})(q^{n-2}-q^{\frac{n-3}{2}})(q^{n-3}+q^{\frac{n-3}{2}})\cdots(q^{n-k+1}-q^{\frac{n-k}{2}})(q^{n-k}+q^{\frac{n-k}{2}})}{(q^{k-1}+q^{\frac{k-1}{2}})(q^{k-2}-q^{\frac{k-3}{2}})(q^{k-3}+q^{\frac{k-3}{2}})\cdots (q-1)\cdot 1}\\
&=\frac{1}{2}q^{\frac{k(n-k)}{2}}\frac{(q^{\frac{n-1}{2}}+1)(q^{\frac{n-1}{2}}-1)(q^{\frac{n-3}{2}}+1)\cdots(q^{\frac{n-k+2}{2}}-1)(q^{\frac{n-k}{2}}+1)}{(q^{\frac{k-1}{2}}+1)(q^{\frac{k-1}{2}}-1)(q^{\frac{k-3}{2}}+1)\cdots (q-1)\cdot 1}\\
&=\frac{1}{2}q^{\frac{k(n-k)}{2}}(q^{\frac{n-k}{2}}+1)\binom{\frac{n-1}{2}}{\frac{k-1}{2}}_{q^{2}}.
\end{align*}
One can derive the expressions of $\binom{n}{k}_{q}^{\perp}$ in other cases similarly, or find them in \cite{Yo4}.
\end{proof}

Recall that the $q$-binomial coefficient $\binom{n}{k}_{q}$ is polynomial in $q$ of degree $k(n-k)$ whose coefficients are non-negative integers. The Euclidean binomial coefficient $\binom{n}{k}_{q}^{\perp}$ also has a similar property. 

\begin{theorem}\label{rational}
The Euclidean binomial coefficient $\binom{n}{k}_{q}^{\perp}$ is a polynomial of degree $k(n-k)$ in the indeterminate $q$. Moreover, $\binom{n}{k}_{q}^{\perp} \in \frac{1}{2}\mathbb{Z}[q]$.
\end{theorem}

\begin{proof}
We use the full description of the Euclidean binomial coefficient $\binom{n}{k}_{q}^{\perp}$ given by \cref{polynomials}. Since $\binom{n}{k}_{q^{2}}$ is always a polynomial in $q$ whose coefficients are non-negative integers, it is not hard to show that $\binom{n}{k}_{q}^{\perp}$ is a polynomial whose coefficients are half-integers using \cref{polynomials} except for the following cases: (1)  $q \equiv 1$ (mod $4$), $n$ and $k$ are even, (2) $q \equiv 3$ (mod $4$), $n \equiv 2 \pmod{4}$ and $k \equiv 2 \pmod{4}$, (3) $q \equiv 3$ (mod $4$), $n \equiv 2 \pmod{4}$ and $k \equiv 0 \pmod{4}$, (4) $q \equiv 3$ (mod $4$), $n \equiv 0 \pmod{4}$ and $k \equiv 2 \pmod{4}$, (5) $q \equiv 3$ (mod $4$), $n \equiv 0 \pmod{4}$ and $k \equiv 0 \pmod{4}$. For these exceptional cases, we rely on \cref{pa}(3). 

\medskip
\noindent Case 1. $q \equiv 1$ (mod $4$), $n$ and $k$ are even:
\[\frac{\left [ n \right ]_{q}^{\perp}-\left [ k \right ]_{q}^{\perp}}{\left [ n-k \right ]_{q}^{\perp}}=\frac{q^{n-1}-q^{\frac{n-2}{2}}-q^{k-1}+q^{\frac{k-2}{2}}}{q^{n-k-1}-q^{\frac{n-k-2}{2}}}=\frac{q^{\frac{k-2}{2}}(q^{\frac{k}{2}}(q^{n-k}-1)-(q^{\frac{n-k}{2}}-1))}{q^{\frac{n-k-2}{2}}(q^{\frac{n-k}{2}}-1)}.\]
Case 2. $q \equiv 3$ (mod $4$), $n \equiv 2 \pmod{4}$ and $k \equiv 2 \pmod{4}$: \[\frac{\left [ n \right ]_{q}^{\perp}-\left [ k \right ]_{q}^{\perp}}{\left [ n-k \right ]_{q}^{\perp}}=\frac{q^{n-1}+q^{\frac{n-2}{2}}-q^{k-1}-q^{\frac{k-2}{2}}}{q^{n-k-1}-q^{\frac{n-k-2}{2}}}=\frac{q^{\frac{k-2}{2}}(q^{\frac{k}{2}}(q^{n-k}-1)+(q^{\frac{n-k}{2}}-1))}{q^{\frac{n-k-2}{2}}(q^{\frac{n-k}{2}}-1)}.\]
Case 3. $q \equiv 3$ (mod $4$), $n \equiv 2 \pmod{4}$ and $k \equiv 0 \pmod{4}$:
\[\frac{\left [ n \right ]_{q}^{\perp}-\left [ k \right ]_{q}^{\perp}}{\left [ n-k \right ]_{q}^{\perp}}=\frac{q^{n-1}+q^{\frac{n-2}{2}}-q^{k-1}+q^{\frac{k-2}{2}}}{q^{n-k-1}+q^{\frac{n-k-2}{2}}}=\frac{q^{\frac{k-2}{2}}(q^{\frac{k}{2}}(q^{n-k}-1)+(q^{\frac{n-k}{2}}+1))}{q^{\frac{n-k-2}{2}}(q^{\frac{n-k}{2}}+1)}.\]
Case 4. $q \equiv 3$ (mod $4$), $n \equiv 0 \pmod{4}$ and $k \equiv 2 \pmod{4}$:
\[\frac{\left [ n \right ]_{q}^{\perp}-\left [ k \right ]_{q}^{\perp}}{\left [ n-k \right ]_{q}^{\perp}}=\frac{q^{n-1}-q^{\frac{n-2}{2}}-q^{k-1}-q^{\frac{k-2}{2}}}{q^{n-k-1}+q^{\frac{n-k-2}{2}}}=\frac{q^{\frac{k-2}{2}}(q^{\frac{k}{2}}(q^{n-k}-1)-(q^{\frac{n-k}{2}}+1))}{q^{\frac{n-k-2}{2}}(q^{\frac{n-k}{2}}+1)}.\]
Case 5. $q \equiv 3$ (mod $4$), $n \equiv 0 \pmod{4}$ and $k \equiv 0 \pmod{4}$:
\[\frac{\left [ n \right ]_{q}^{\perp}-\left [ k \right ]_{q}^{\perp}}{\left [ n-k \right ]_{q}^{\perp}}=\frac{q^{n-1}-q^{\frac{n-2}{2}}-q^{k-1}+q^{\frac{k-2}{2}}}{q^{n-k-1}-q^{\frac{n-k-2}{2}}}=\frac{q^{\frac{k-2}{2}}(q^{\frac{k}{2}}(q^{n-k}-1)-(q^{\frac{n-k}{2}}-1))}{q^{\frac{n-k-2}{2}}(q^{\frac{n-k}{2}}-1)}.\]
By \cref{pa}(2), we may assume that $q^{\frac{k-2}{2}} \geq q^{\frac{n-k-2}{2}}$, which is equivalent to $k \geq n/2$. 
Thus $\left [ n \right ]_{q}^{\perp}-\left [ k \right ]_{q}^{\perp}/\left [ n-k \right ]_{q}^{\perp}$ are polynomials in $q$ whose coefficients are integers. 
Since $\binom{n-1}{k-1}_{q}^{\perp}$ and $\binom{n-1}{k}_{q}^{\perp}$ are polynomials in $q$ whose coefficients are half-integers, we conclude our result. 
Let us now observe the highest degree in each term to compute the degree of the Euclidean binomial coefficient $\binom{n}{k}_{q}^{\perp}$. 
The degree of the Euclidean binomial coefficient $\binom{n}{k}_{q}^{\perp}$ is shown to be
\[\sum_{i=1}^{k}(n-i)-\sum_{i=1}^{k-1}i=k(n-k).\]
regardless of any cases. 
\end{proof}

\begin{example}
Let us compute $\binom{13}{5}_{q}^{\perp}$ when $q \equiv 1$ (mod $4$). 
Note that

\begin{align*}
\binom{13}{5}_{q}^{\perp}&=\frac{(q^{12}+q^{6})(q^{11}-q^{5})(q^{10}+q^{5})(q^{9}-q^{4})(q^{8}+q^{4})}{2(q^{4}+q^{2})(q^{3}-q)(q^{2}+q)(q-1)}\\
&=\frac{1}{2}q^{20}(q^{4}-q^{2}+1)(q^{4}+q^{2}+1)(q^{4}-q^{3}+q^{2}-q+1)(q^{4}+q^{3}+q^{2}+q+1)(q^{4}+1).
\end{align*}
Thus, $\binom{13}{5}_{q}^{\perp}$ is a polynomial of degree $40$ in $q$.
\end{example}

\begin{remark}
(1) The Euclidean binomial coefficient $\binom{n}{k}_{q}^{\perp}$ can be written as
\[\binom{n}{k}_{q}^{\perp}=(1+o(1))\frac{q^{k(n-k)}}{2},\]
where $o(1)$ goes to zero as $q$ goes to infinity.

\smallskip
(2) In any case, we obtain \[\lim_{q\rightarrow \infty}\frac{\binom{n}{k}_{q}^{\perp}}{\binom{n}{k}_{q}}=\frac{1}{2}.\]
Thus, in $(\mathbb{F}_{q}^{n},\text{Euc}_{n})$, the half of subspaces is roughly $k$-dimensional Euclidean subspaces. 
One can see that roughly the other half consists of $k$-dimensional Lorentzian subspaces by considering the counting formula given in \cref{section4}.

\end{remark}

\subsection{Symmetries of the polynomial coefficients of $\binom{n}{k}_{q}^{\perp}$}
By the expressions given by the complete list of the Euclidean binomial coefficient $\binom{n}{k}_{q}^{\perp}$ in \cite{Yo4}, we obtain the following expressions: If (1) $q \equiv 1$ (mod $4$), $n$ is even and $k$ is odd, or (2) $q \equiv 3$ (mod $4$), $n \equiv 2 \pmod{4}$ and $k$ is odd, or (3) $n \equiv 0 \pmod{4}$ and $k$ is odd, then we have
\begin{equation}\label{expression1}
    \binom{n}{k}_{q}^{\perp}=q^{\frac{k(n-k)-1}{2}}\sum_{i=0}^{\frac{k(n-k)+1}{2}}a_{i}q^{i}.
\end{equation}

Otherwise, we have
\begin{equation}\label{expression2}
 \binom{n}{k}_{q}^{\perp}=q^{\frac{k(n-k)}{2}}\sum_{i=0}^{\frac{k(n-k)}{2}}b_{i}q^{i}.   
\end{equation}
We show that the coefficients $a_{i}$ in \cref{expression1} and $b_{i}$ in \cref{expression2} have some sort of symmetries.
A necessary lemma to prove our claim is the following.

\begin{lemma}\label{poly}
Suppose that $q \equiv 1$ \textup{(mod $4$)}. Then
\[\binom{n}{k}_{\frac{1}{q}}^{\perp} = \begin{cases}
\quad \frac{-1}{q^{\frac{3k(n-k)-1}{2}}}\binom{n}{k}_{q}^{\perp} & \text{ if }  n \text{ is even and }k \text{ is odd,}\\ 
\quad\frac{1}{q^{\frac{3k(n-k)}{2}}}\binom{n}{k}_{q}^{\perp} & \text{ otherwise }  
\end{cases}.\]
Suppose that $q \equiv 3$ \textup{(mod $4$)}. Then
\[\binom{n}{k}_{\frac{1}{q}}^{\perp} = \begin{cases}
\quad \frac{-1}{q^{\frac{3k(n-k)-1}{2}}}\binom{n}{k}_{q}^{\perp} & \text{ if } n \equiv 0 \pmod{4}\text{ and }k\text{ is odd}, \\ 
\quad \frac{1}{q^{\frac{3k(n-k)-1}{2}}}\binom{n}{k}_{q}^{\perp} & \text{ if } n \equiv 2 \pmod{4}\text{ and }k\text{ is odd},\\
\quad \frac{-1}{q^{\frac{3k(n-k)}{2}}}\binom{n}{k}_{q}^{\perp} & \text{ if } n \equiv 1 \pmod{4} \text{ and }k \equiv 2,3 \pmod{4},  \\ 
& \quad  \text{ or }~~ n \equiv 3 \pmod{4}\text{ and }k \equiv 1,2 \pmod{4}, \\
\quad  \frac{1}{q^{\frac{3k(n-k)}{2}}}\binom{n}{k}_{q}^{\perp} & \text{ otherwise }  
.
\end{cases}\]
\end{lemma}

\begin{proof}
Suppose that $q \equiv 1$ (mod $4$), and $n$ and $k$ are odd. By the expression of the Euclidean binomial coefficient $\binom{n}{k}_{q}^{\perp}$, we have
\begin{align*}
\binom{n}{k}_{\frac{1}{q}}^{\perp}
&=\frac{1}{2}\frac{1}{q^{\frac{k(n-k)}{2}}}\frac{\left(\frac{1}{q^{\frac{n-1}{2}}}+1\right)\left(\frac{1}{q^{\frac{n-1}{2}}}-1\right)\left(\frac{1}{q^{\frac{n-3}{2}}}+1\right)\cdots\left(\frac{1}{q^{\frac{n-k+2}{2}}}-1\right)\left(\frac{1}{q^{\frac{n-k}{2}}}+1\right)}{\left(\frac{1}{q^{\frac{k-1}{2}}}+1\right)\left(\frac{1}{q^{\frac{k-1}{2}}}-1\right)\left(\frac{1}{q^{\frac{k-3}{2}}}+1\right)\cdots \left(\frac{1}{q}-1\right)\cdot 1}\\
&=\frac{1}{2}\frac{1}{q^{\frac{k(n-k)}{2}}}\frac{q^{\frac{k-1}{2}}q^{\frac{k-1}{2}}q^{\frac{k-3}{2}}\cdots q}{q^{\frac{n-1}{2}}q^{\frac{n-1}{2}}q^{\frac{n-3}{2}}\cdots q^{\frac{n-k}{2}}}\frac{(q^{\frac{n-1}{2}}+1)(q^{\frac{n-1}{2}}-1)(q^{\frac{n-3}{2}}+1)\cdots(q^{\frac{n-k+2}{2}}-1)(q^{\frac{n-k}{2}}+1)}{(q^{\frac{k-1}{2}}+1)(q^{\frac{k-1}{2}}-1)(q^{\frac{k-3}{2}}+1)\cdots (q-1)\cdot 1}\\
&=\frac{1}{q^{\frac{3k(n-k)}{2}}}\binom{n}{k}_{q}^{\perp}.
\end{align*}
Other cases can be proved similarly. 
\end{proof}

\begin{corollary}\label{symmetry}
Suppose that $q \equiv 1$ \textup{(mod $4$)}. For the coefficients $a_{i}$ in \cref{expression1} and $b_{i}$ in \cref{expression2}, \begin{enumerate}
    \item[\textup{(1)}] $a_{i}=-a_{(k(n-k)+1)/2-i}$ (anti-symmetric) if $n$ is even and $k$ is odd.
    \item[\textup{(2)}] $b_{i}=b_{k(n-k)/2-i}$ (symmetric) otherwise.

\end{enumerate}
Suppose that $q \equiv 3$ \textup{(mod $4$)}. Then, we have
\begin{enumerate}
    \item[\textup{(1)}] $a_{i}=-a_{(k(n-k)+1)/2-i}$ (anti-symmetric) if $n \equiv 0 \pmod{4}$ and $k$ is odd,
    \item[\textup{(2)}] $a_{i}=a_{(k(n-k)+1)/2-i}$ (symmetric) if $n \equiv 2 \pmod{4}$ and $k$ is odd,
    \item[\textup{(3)}] $b_{i}=-b_{k(n-k)/2-i}$ (anti-symmetric) if $n \equiv 1 \pmod{4}$ and $k \equiv 2,3 \pmod{4}$, or $n \equiv 3 \pmod{4}$ and $k \equiv 1,2 \pmod{4}$,    
    \item[\textup{(4)}] $b_{i}=b_{k(n-k)/2-i}$ (symmetric) otherwise.
\end{enumerate}
\end{corollary}

\begin{proof}
Let $q \equiv 1$ (mod $4$), $n$ be odd, and $k$ be odd. Let us denote $\binom{n}{k}_{q}^{\perp}$ by
\[\binom{n}{k}_{q}^{\perp}=\frac{1}{2}q^{\frac{k(n-k)}{2}}(b_{0}+b_{1}q+\cdots+b_{\frac{k(n-k)}{2}}q^{\frac{k(n-k)}{2}}).\]
Then, we have
\begin{align*}
\binom{n}{k}_{\frac{1}{q}}^{\perp}&=\frac{1}{2}\frac{1}{q^{\frac{k(n-k)}{2}}}(b_{0}+\frac{b_{1}}{q}+\cdots+\frac{b_{\frac{k(n-k)}{2}}}{q^{\frac{k(n-k)}{2}}})\\
&=\frac{1}{q^{\frac{3k(n-k)}{2}}}\frac{1}{2}q^{\frac{k(n-k)}{2}}(b_{0}q^{\frac{k(n-k)}{2}}+b_{1}q^{\frac{k(n-k)}{2}-1}+\cdots+b_{\frac{k(n-k)}{2}})
\end{align*}
Thus, we only need to show
\[\binom{n}{k}_{\frac{1}{q}}^{\perp}=\frac{1}{q^{\frac{3k(n-k)}{2}}}\binom{n}{k}_{q}^{\perp},\]
which is given by \cref{poly}. Other cases can be obtained by a similar way.
\end{proof}

\subsection{Combinatorial correspondence between symmetric sets and Euclidean-analogues}\label{symmetric}

Notice that $ \lim_{q \rightarrow 1}\binom{n}{k}_{q}=\binom{n}{k}$ reveals the connection between binomial coefficients and $q$-binomial coefficients. Thus, to find the correspondence between sets and the Euclidean analogues, the first step is to compute $\lim_{q \rightarrow 1}\binom{n}{k}_{q}^{\perp}$. However, $\lim_{q \rightarrow 1}\binom{n}{k}_{q}^{\perp}$ when $q \equiv 1$ (mod $4$) is not always the same as $\lim_{q \rightarrow 1}\binom{n}{k}_{q}^{\perp}$ when $q \equiv 3$ (mod $4$). For example, if $q \equiv 1$ (mod $4$), $n$ is odd and $k$ is even, we have
\[\lim_{q \rightarrow 1} \binom{n}{k}_{q}^{\perp}=\lim_{q \rightarrow 1}\left (\frac{1}{2}q^{\frac{k(n-k)}{2}}(q^{\frac{k}{2}}+1)\binom{\frac{n-1}{2}}{\frac{k}{2}}_{q^{2}}\right )=\binom{\frac{n-1}{2}}{\frac{k}{2}}.\]
On the other hand, if $q \equiv 3$ (mod $4$), $n \equiv 3 \pmod{4}$, and $k \equiv 2 \pmod{4}$, we obtain
\[\lim_{q \rightarrow 1} \binom{n}{k}_{q}^{\perp}=\lim_{q \rightarrow 1}\left (\frac{1}{2}q^{\frac{k(n-k)}{2}}(q^{\frac{k}{2}}-1)\binom{\frac{n-1}{2}}{\frac{k}{2}}_{q^{2}}\right )=0.\]

Thus, we attempt to compute $\lim_{q \rightarrow -1}\binom{n}{k}_{q}^{\perp}$ when $q \equiv 3$ (mod $4$). 
It turns out that the limits $\lim_{q \rightarrow 1} \binom{n}{k}_{q}^{\perp}$ when $q \equiv 1$ (mod $4$) and $\lim_{q \rightarrow -1}\binom{n}{k}_{q}^{\perp}$ when $q\equiv 3$ (mod $4$) are the same. 

\begin{lemma}\label{limit}
For any $1 \le k \le n$, we have 
\begin{enumerate}
    \item[\textup{(1)}] If $q \equiv 1$ \textup{(mod $4$)}, $\lim_{q \to 1}\binom{n}{k}_{q}^{\perp}=\binom{\left \lfloor n/2 \right \rfloor}{\left \lfloor k/2 \right \rfloor}$, and $0$ if $n$ is even and $k$ is odd,
    \item[\textup{(2)}] If $q \equiv 3$ \textup{(mod $4$)}, $\lim_{q \to -1}\binom{n}{k}_{q}^{\perp}=\binom{\left \lfloor n/2 \right \rfloor}{\left \lfloor k/2 \right \rfloor}$, and $0$ if $n$ is even and $k$ is odd.
\end{enumerate}
\end{lemma}

\begin{proof}
Notice that
\[\lim_{q \rightarrow 1}\binom{n}{k}_{q^{2}}=\binom{n}{k} \text{ and }\lim_{q \rightarrow -1}\binom{n}{k}_{q^{2}}=\binom{n}{k}.\] 
The limits can be directly computed by the expressions in \cref{polynomials}.  \end{proof}

\cref{limit} indicates that combinatorial descriptions of sets corresponding to the Euclidean binomial coefficient $\binom{n}{k}_{q}^{\perp}$ do not need to depend on whether $q \equiv 1$ (mod $4$) or $q \equiv 3$ (mod $4$). We now show that the limit of $\binom{n}{k}_{q}^{\perp}$ in \cref{limit} is the number of symmetric $k$-sets in $\mathbb{Z}/(n+1)\mathbb{Z}$. Here,  we call a subset $A$ of $\mathbb{Z}/(n+1)\mathbb{Z}$ a \textbf{symmetric} $k$-set if $|A|=k$, $0 \notin A$, and $A=-A$.

\begin{theorem} \label{symm} For any $1 \le k \le n$, the limit of the Euclidean binomial coefficient $\binom{n}{k}_{q}^{\perp}$ is the number of symmetric $k$-sets in $\mathbb{Z}/(n+1)\mathbb{Z}$. 
\end{theorem}

\begin{proof}
It suffices to show that the limit in \cref{limit} is the number of symmetric $k$-sets in $\mathbb{Z}/(n+1)\mathbb{Z}$ for each case.
Suppose that $n$ and $k$ are odd. Then, we have \[\mathbb{Z}/(n+1)\mathbb{Z}=\left \{ 0,1,\ldots,\frac{n-1}{2},\frac{n+1}{2},\frac{n+3}{2}\ldots,n \right \}.\]
To make a symmetric $k$-set, we only need to choose $(k-1)/2$ elements in $\left \{ 1,\ldots,(n-1)/2 \right \}$. Then, the other $(k-1)/2$ elements are determined immediately. By adding the element $(n+1)/2$ in the set, we complete it to make a symmetric $k$-set. Thus, there are $\binom{(n-1)/2}{(k-1)/2}$ ways to make symmetric $k$-sets in $\mathbb{Z}/(n+1)\mathbb{Z}$. Similarly, the number of symmetric $k$-sets in $\mathbb{Z}/(n+1)\mathbb{Z}$ is $\binom{(n-1)/2}{k/2}$ if $n$ is odd and $k$ is even, and the number of symmetric $k$-sets in $\mathbb{Z}/(n+1)\mathbb{Z}$ is $\binom{n/2}{k/2}$ if $n$ and $k$ are even. When $n$ is even and $k$ is odd, it is not hard to check that symmetric $k$-sets do not exist.
\end{proof}

\begin{example}
There are six symmetric $4$-sets in $\mathbb{Z}/9\mathbb{Z}$ as follows:
\[\left\{ 1,2,7,8 \right\},\left\{ 1,3,6,8 \right\},\left\{ 1,4,5,8 \right\},\left\{ 2,3,6,7 \right\},\left\{ 2,4,5,7 \right\},\left\{ 3,4,5,6 \right\}.\]
By \cref{symm}, the number of symmetric $4$-sets in $\mathbb{Z}/9\mathbb{Z}$ is also given by
\begin{align*}
    \lim_{q\rightarrow 1}\binom{8}{4}_{q}^{\perp}&=\lim_{q\rightarrow 1}\frac{(q^{7}-q^{3})(q^{6}+q^{3})(q^{5}-q^{2})(q^{4}+q^{2})}{2(q^{2}-1)(q+1)(q-1)}\\&=\frac{1}{2}\lim_{q\rightarrow 1}q^{8}(q^{2}+1)^{2}(q^{2}-q+1)(q^{2}+q+1)=6.
\end{align*}
\end{example}

Next, for completeness, we give the limit of $\binom{n}{k}_{q}^{\perp}$ as $q$ goes to $-1$ when $q \equiv 1$ (mod $4$) and as $q$ goes to $1$ when $q \equiv 3$ (mod $4$).
We do not know any combinatorial description of this limit.
The proof is straightforward using \cref{polynomials}, and thus omitted.

\begin{proposition}\label{limit2}
For any $1 \le k<n$, we have 
\begin{enumerate}
    \item[\textup{(1)}] If $q \equiv 1$ \textup{(mod $4$)}, 
    \[\lim_{q \to -1}\binom{n}{k}_{q}^{\perp}=\begin{cases} ~~
    \binom{\left \lfloor n/2 \right \rfloor}{\left \lfloor k/2 \right \rfloor}
 &  \text{ if } n \equiv 1 ~(\textup{mod }4), ~~ k \equiv 1 ~ (\textup{mod }4) \text{ or } k \equiv 0 ~ (\textup{mod }4),  \\
 & \quad  n \equiv 3 ~(\textup{mod }4),~k \equiv 3 ~(\textup{mod }4)\text{ or }k \equiv 0 ~(\textup{mod }4), \\
 & \quad n \equiv 0 ~(\textup{mod }4), ~k \equiv 0 ~(\textup{mod }4), \\
~~\binom{ n/2  -1}{\left \lfloor k/2 \right \rfloor} & \text{ if } n \equiv 2 ~(\textup{mod }4), ~k \equiv 1 ~(\textup{mod }4),\text{ or }k \equiv 3 ~(\textup{mod }4) \text{ or } k \equiv 0 ~(\textup{mod }4), \\
~~\binom{ n/2  -1}{ k/2  -1} & \text{ if } n \equiv 2 ~(\textup{mod }4),~k \equiv 2 ~(\textup{mod }4), \\
\; \; 0 & \text{ otherwise } .
\end{cases}\]

    \item[\textup{(2)}] If $q \equiv 3$ \textup{(mod $4$)}, $\lim_{q \to 1}\binom{n}{k}_{q}^{\perp}$is the same with $\lim_{q \to -1}\binom{n}{k}_{q}^{\perp}$ when $q \equiv 1$ \textup{(mod $4$)}.
\end{enumerate}
\end{proposition}


\subsection{The Euclidean flags and the orthogonal group $O(n,q)$}  

\begin{definition}
Let $V$ be a $n$-dimensional Euclidean space over $\mathbb{F}_{q}$. A \textbf{Euclidean flag} of length $m$ is a sequence of Euclidean subspaces of $V$ as follows:
\[\{0\}=V_{0} \subset V_{1} \subset \cdots \subset V_{m}= V.\]
Let $n_{k}=\text{dim}V_{k}$, where $k=1,2,\ldots,m$. Then, we have $k \le n$. Let us call $(n_{1},n_{2},\ldots,n_{m})$ the \textbf{signature} of a Euclidean flag. A \textbf{complete Euclidean flag} of $V$ is a Euclidean flag of length $n$. In this case, we have $n_{k}=k$ for any $i=0,1,\ldots,n$.
\end{definition}

We count the number of flags of length $m$ with signature $(n_{1},n_{2},\ldots,n_{m})$. Since dim$V_{1}=n_{1}$, there are $\binom{n}{n_{1}}_{q}^{\perp}$ many choices for $V_{1}$. Similarly, there are $\binom{n-n_{1}}{n_{2}}_{q}^{\perp}$ choices for $V_{2}$. By repeating this process, the number of flags of length $m$ with the signature $(n_{1},n_{2},\ldots,n_{m})$ is given by
\[\binom{n}{n_{1}}_{q}^{\perp}\binom{n-n_{1}}{n_{2}}_{q}^{\perp}\cdots\binom{n-n_{1}-\cdots-n_{m-1}}{n_{m}}_{q}^{\perp}.
\]
We give a new terminology for this count. 

\begin{definition}
Let $n$ be a positive integer with $n=\sum_{i=1}^{m}n_{i}$, where $n_{i}$ is a positive integer for each $i=1,\ldots,m$. We define the \textbf{Euclidean multinomial coefficient} as follows:
\[\binom{n}{n_{1}~n_{2}~\cdots ~n_{m}}_{q}^{\perp}:=\frac{\left [ n \right ]_{q}^{\perp}!}{\left [ n_{1} \right ]_{q}^{\perp}!\left [ n_{2} \right ]_{q}^{\perp}!\cdots \left [ n_{m} \right ]_{q}^{\perp}!}=\binom{n}{n_{1}}_{q}^{\perp}\binom{n-n_{1}}{n_{2}}_{q}^{\perp}\cdots\binom{n-n_{1}-\cdots-n_{m-1}}{n_{m}}_{q}^{\perp}.\]
\end{definition}
Since a complete Euclidean flag is a Euclidean flag of length $n$, the number of complete Euclidean flags is $\left [  n \right ]_{q}^{\perp}!$. Next, we find a way to count the size of orthogonal groups using the number of complete Euclidean flags in $E_{n}(q)$. 

\begin{theorem}\label{sizeofog}
The size of the orthogonal group $O(n,q)$ over $\mathbb{F}_{q}$ with \textup{Euc}$_{n}$ is given as follows:
\begin{equation}\label{og}
    |O(n,q)|=2^{n}\left [ n \right ]_{q}^{\perp}!.
\end{equation}
\end{theorem}

\begin{proof}
Note that the number of Euclidean flags of length $n$ in $E_{n}(q)$ is $\left [ n \right ]_{q}^{\perp}!$. We provide another way to find the number of the Euclidean flags in $E_{n}(q)$. Given a Euclidean flag with an orthonormal basis $\left \{e_{1},\ldots,e_{n} \right \}$
\[\text{span}(e_{1}) \subset \text{span}(e_{1},e_{2})\subset \cdots \subset \text{span}(e_{1},e_{2},\ldots,e_{n}),\]
the flag is unchanged if we change $e_{i}$ to $-e_{i}$ for any $i$, which yields $2^{n}$ cases. Thus Euclidean flags are bijective up to a factor of $2^{n}$ with orthonormal bases. Since the orthogonal group $O(n,q)$ over $\mathbb{F}_{q}$ with Euc$_{n}$ is bijective with orthonormal bases of $(\mathbb{F}_{q}^{n},\text{Euc}_{n})$, we obtain
\begin{align*}
\left [ n \right ]_{q}^{\perp}!&=\text{the number of the Euclidean flags}\\
&=\text{the number of orthonormal bases up to }\pm\\
&=\frac{|O(n,q)|}{2^{n}},
\end{align*}
which yields our desired result. 
\end{proof}
The formula of the size of orthogonal groups over finite fields in \cref{sizeofog} matches with the size of symmetric groups $S_{n}$, and general linear groups over finite fields $GL(n,q)$ using brackets as follows:
$|S_{n}|=n! \quad \textup{and} \quad |GL(n,q)|=q^{n(n-1)/2}(q-1)^{n}[n]_{q}!.$
Furthermore, this formula redrives a usual formula for the size of the orthogonal group $O(n,q)$. 
\begin{corollary}\cite{Ta}\label{ogf}
For any $n \geq 3$, the size of the orthogonal group $O(n,q)$ over $\mathbb{F}_{q}$ with \textup{Euc}$_{n}$ is given by the following: if $n=2k+1$, then we have
\[|O(2k+1,q)|=2q^{k^{2}}\prod^{k}_{i=1}\big(q^{2i}-1\big). \]
If $n=2k$, and $q \equiv 1$ \textup{(mod $4$)}, we have
\[|O(2k,q)|=2q^{k(k-1)}\big(q^{k}-1\big)\prod^{k-1}_{i=1}\big(q^{2i}-1\big).\]
If $n=2k$, and $q \equiv 3$ \textup{(mod $4$)}, we have
\[|O(2k,q)|=2q^{k(k-1)}\big(q^{k}+1\big)\prod^{k-1}_{i=1}\big(q^{2i}-1\big).\]
\end{corollary}
\begin{proof}
Let $n=2k+1$ and $q \equiv 1$ (mod $4$). By \cref{og} and \cref{dotline}, 
\begin{align*}
    |O(2k+1,q)|&=2^{2k+1}[2k+1]_{q}^{\perp}!\\
    &=2^{2k+1}\cdot1\cdot \Big(\frac{q-1}{2} \Big) \Big ( \frac{q^{2}+q}{2} \Big) \Big( \frac{q^{3}-q}{2}\Big) \Big(\frac{q^{4}+q^{2}}{2} \Big)\cdots  \Big(\frac{q^{2k-1}-q^{k-1}}{2} \Big)\Big( \frac{q^{2k}+q^{k}}{2} \Big)\\
    &=2q(q-1)(q+1)q(q^{3}-q)(q^{3}+q)\cdots q(q^{2k-1}-q^{k-1})(q^{2k-1}+q^{k-1})\\
    &=2q^{k}(q^{2}-1)(q^{6}-q^{2})\cdots(q^{2(2k-1)}-q^{2(k-1)})\\
    &=2q^{k}q^{2}\cdots q^{2(k-1)}(q^{2}-1)(q^{4}-1)\cdots (q^{2k}-1)\\
    &=2q^{k^{2}}\prod^{k}_{i=1}\big(q^{2i}-1\big).
\end{align*}
Other cases can be computed similarly.
\end{proof}

Let us discuss how to count the Euclidean binomial coefficient $\binom{n}{k}_{q}^{\perp}$ in other ways. By using the number of the Euclidean flags, we obtain
\begin{equation}\label{oo}
\binom{n}{k}_{q}^{\perp}=\frac{\left [ n \right ]_{q}^{\perp}!}{\left [ k \right ]_{q}^{\perp}!\left [ n-k \right ]_{q}^{\perp}!}=\frac{|O(n,q)|}{|O(k,q)\times O(n-k,q)|}\cdot \frac{2^{k}\cdot 2^{n-k}}{2^{n}}=\left | \frac{O(n,q)}{O(k,q)\times O(n-k,q)} \right |.
\end{equation}
On the other hand, the orthogonal group $O(n,q)$ acts transitively on $Gr_{q}^{\perp}(k,n)$ by Witt's extension theorem. Then, the stabilizer of any element in $Gr_{q}^{\perp}(k,n)$ is $O(k,q)\times O(n-k,q)$. Hence, there is an isomorphism between $Gr_{q}^{\perp}(k,n)$ and $\frac{O(n,q)}{O(k,q)\times O(n-k,q)}$. Therefore, we have another way to obtain \cref{oo} by the theory of group actions.

\subsection{The M\"{o}bius function of $E_{n}(q)$}
Let $P=(X,\le)$ be a poset. 
The M\"{o}bius function $\mu:X\times X \rightarrow \mathbb{R}$ of $P$ is defined by 
\[\mu(x,y)=\begin{cases}
1 & \text{ if } x= y \\ 
-\sum_{x \leq z < y}\mu(x,z) & \text{ if } x \ne  y 
\end{cases}.\]

As examples, one can find the M\"{o}bius functions of the Boolean algebra $B_{n}$ and the poset $L_{n}(q)$ constructed by subspaces of $\mathbb{F}_{q}^{n}$ under set-inclusion. These two posets have the least element and a nice property, called the \textbf{homogeneity property}: for any $x,y$ with $x \leq y$, there is an element
$z$ such that the interval $\left [x,y  \right ]$ is isomorphic to the interval $\left [0,c  \right ]$ (with $\left [x,y  \right ]$ and $\left [0,c  \right ]$ as posets), where $0$ is the least element of the poset. By these properties, we have $\mu(x,y) = \mu (0, c)$. This helps us to reduce the complexity of the computations for the M\"{o}bius functions. 
Recall that the Euclidean poset $E_{n}(q)$ is the poset constructed by Euclidean subspaces of $(\mathbb{F}_{q}^{n},\textup{Euc}_{n})$ under set-inclusion.
We notice that the Euclidean poset $E_{n}(q)$ also has the homogeneity property. This means that, for two Euclidean subspaces $U$ and $W$ of $\mathbb{F}_{q}^{n}$ with $U \leq W$, the interval $\left [ 
U,W \right ]$ is isomorphic to the poset of $W/U$. 
Thus, we can consider the M\"{o}bius function of $E_{n}(q)$ as a one-variable function.
More discussions on this topic can be found in \cite{Ca2, St3}.

Note that $B_{n}$ only depends on the cardinality of subsets because the symmetric group $S_{n}$ acts transitively on the set of subsets whose cardinalities are the same. Similarly, $L_{n}(q)$ only depends on the dimension of subspaces since the general linear group $GL(n,q)$ acts transitively on the set of subspaces of the same dimension. We provide the M\"{o}bius functions of $B_{n}$ and $L_{n}(q)$ as follows:

\begin{theorem}\textup{\cite{Ca2}}
The M\"{o}bius functions of $B_{n}$ and $L_{n}(q)$ are the following:
\begin{enumerate}
\item[\textup{(1)}] Given a subset $X$ of $\left \{ 1,2,\ldots,n \right \}$, we have $\mu(\left \{0  \right \},X)=(-1)^{|X|}$,
\item[\textup{(2)}] Given a $k$-dimensional subspace $V$ of $\mathbb{F}_{q}^{n}$, we have $\mu(\left \{0  \right \},V)=(-1)^{k}q^{\frac{k(k-1)}{2}}$.
\end{enumerate}
\end{theorem}
The key ideas are to use the binomial theorem and the $q$-binomial theorem. To compute the M\"{o}bius function of the Euclidean poset $E_{n}(q)$, we need the right binomial theorem for the Euclidean analogues, which we do not have. Moreover, there are some difficulties to obtain the explicit formula; thus we just give an algorithm to compute the M\"{o}bius function of $E_{n}(q)$. In the following proposition, \cref{0} plays the role of the binomial theorem for Euclidean analogues.  

\begin{proposition}\label{mobi}
Let $V$ be a $n$-dimensional Euclidean space. The M\"{o}bius function of $E_{n}(q)$ is given by
\[
\mu(\left \{0  \right \},V)=(-1)^{n}b_{n},
\] 
where $b_{0}=1$ and $b_{n}$ can be obtained recursively by the equation
\begin{equation}\label{0}\sum_{k=0}^{n}(-1)^{k}b_{k}\binom{n}{k}_{q}^{\perp}=0.
\end{equation}
\end{proposition}
\begin{proof}
We prove it by induction. If $n=0$, then $\mu(\left \{ 0 \right \},V)=(-1)^{0}b_{0}=1$. Suppose that the statement is true for $k<n$. Since there are $\binom{n}{k}_{q}^{\perp}$ many $k$-dimensional Euclidean subspaces in $V$, we obtain
\begin{align*}
\mu(0,V)&=-\sum_{0\leq W<V}\mu(0,W)\\
&=-\left ( \sum_{0\leq W<V}(-1)^{\text{dim}W}b_{\text{dim}W} \right )\\
&=-\left( \sum_{k=0}^{n-1}(-1)^{k}b_{k}\binom{n}{k}_{q}^{\perp}\right)\\
&=-\left ( \sum_{k=0}^{n}(-1)^{k}b_{k}\binom{n}{k}_{q}^{\perp} -(-1)^{n}b_{n}\right )\\
&=(-1)^{n}b_{n}
\end{align*}
by the induction hypothesis and \cref{0}.
\end{proof}

Note that $E_{n}(q)$ only depends on the dimension of subspaces since and the orthogonal group $O(n,q)$ acts transitively on the set of Euclidean subspaces that have the same dimension.

\begin{example}\label{b}
Let us compute the M\"{o}bius function for $E_{3}(q)$ if $q \equiv 1$ (mod $4$). Recursively, by \cref{dotb}, \cref{0}, and \cref{dotline}, we obtain
\begin{align*}
b_{1}&=b_{0}=1;\\ 
b_{2}&=-1+\binom{2}{1}_{q}^{\perp}=-1+\frac{q-1}{2}=\frac{q-3}{2};\\
b_{3}&=1-\left (  \binom{3}{1}_{q}^{\perp}+\binom{3}{2}_{q}^{\perp}\right )+\binom{2}{1}_{q}^{\perp}\binom{3}{2}_{q}^{\perp}=1-(q^{2}+q)+\frac{q^{3}-q}{4}.
\end{align*}
Therefore, we have $\mu(\left \{ 0 \right \},\text{Euc}_{3})=(-1)^{3}b_{3}=-1+(q^{2}+q)-\frac{q^{3}-1}{4}.$
\end{example}

\section{The Lorentzian poset $LO_{n}(q)$ and other types}\label{section4}

\subsection{Counting formulas and their limits}\label{other limits}
Let us introduce some terminologies: 
\begin{align*}
&~(1) ~ |\text{Euc}_{k},\text{Euc}_{n}|_{q}:=\text{the number of the set of $k$-dimensional Euclidean subspaces of } (\mathbb{F}_{q}^{n},\text{Euc}_{n}),\\
&~(2) ~ |\text{Lor}_{k},\text{Euc}_{n}|_{q}:=\text{the number of the set of $k$-dimensional Lorentzian }\text{subspaces of } (\mathbb{F}_{q}^{n},\text{Euc}_{n}),\\
&~(3) ~ |\text{Euc}_{k},\text{Lor}_{n}|_{q}:=\text{the number of the set of $k$-dimensional Euclidean }\text{subspaces of }(\mathbb{F}_{q}^{n},\text{Lor}_{n}),\\
&~(4) ~ |\text{Lor}_{k},\text{Lor}_{n}|_{q}:=\text{the number of the set of $k$-dimensional Lorentzian }\text{subspaces of }(\mathbb{F}_{q}^{n},\text{Lor}_{n}).
\end{align*}

For completeness, we summarize all analogues of binomial coefficients and omit details since the strategies of the proofs are similar to $\binom{n}{k}_{q}^{\perp}$. 
\begin{theorem}\label{other counts} For $1 \le k < n$, we have
\begin{align*} 
&\binom{n}{k}_{q}^{\perp}:=
|\textup{Euc}_{k},\textup{Euc}_{n}|_{q}=\frac{|\textup{Euc}_{1},\textup{Euc}_{n}|_{q}|\textup{Euc}_{1},\textup{Euc}_{n-1}|_{q}\cdots|\textup{Euc}_{1},\textup{Euc}_{n-k+1}|_{q}}{|\textup{Euc}_{1},\textup{Euc}_{k}|_{q}\cdots|\textup{Euc}_{1},\textup{Euc}_{1}|_{q}},\\
&\binom{\overline{n}}{k}_{q}^{\perp}:=
|\textup{Euc}_{k},\textup{Lor}_{n}|_{q}=\frac{|\textup{Euc}_{1},\textup{Lor}_{n}|_{q}|\textup{Euc}_{1},\textup{Lor}_{n-1}|_{q}\cdots|\textup{Euc}_{1},\textup{Lor}_{n-k+1}|_{q}}{|\textup{Euc}_{1},\textup{Euc}_{k}|_{q}\cdots|\textup{Euc}_{1},\textup{Euc}_{1}|_{q}},\\
&\binom{n}{\overline{k}}_{q}^{\perp}:=|\textup{Lor}_{k},\textup{Euc}_{n}|_{q}=\frac{|\textup{Lor}_{1},\textup{Euc}_{n}|_{q}}{|\textup{Lor}_{1},\textup{Lor}_{k}|_{q}}\binom{\overline{n-1}}{k-1}_{q}^{\perp},\\
&\binom{\overline{n}}{\overline{k}}_{q}^{\perp}:=|\textup{Lor}_{k},\textup{Lor}_{n}|_{q}=\frac{|\textup{Lor}_{1},\textup{Lor}_{n}|_{q}}{|\textup{Lor}_{1},\textup{Lor}_{k}|_{q}}\binom{n-1}{k-1}_{q}^{\perp}.
\end{align*}
\end{theorem}
We adopt the convention that $|\text{Lor}_{0}, \text{Euc}_{n}|_{q}=|\text{Euc}_{0}, \text{Lor}_{n}|_{q}=|\text{Lor}_{0}, \text{Lor}_{n}|_{q}=1$.
Recall that the counts $|\text{Euc}_{1}, \text{Euc}_{n}|_{q}$, $|\text{Lor}_{1}, \text{Euc}_{n}|_{q}$, $|\text{Euc}_{1}, \text{Lor}_{n}|_{q}$, and $|\text{Lor}_{1}, \text{Lor}_{n}|_{q}$ are discussed in \cref{dotline} and \cref{lambdadot}. 
As an example of \cref{other counts}, we revisit \cref{all2} and drive the number of $2$-dimensional Lorentzian subspaces of $(\mathbb{F}_{3}^{3},\text{Euc}_{3})$ computing
\[\binom{3}{\overline{2}}_{3}^{\perp}=\frac{|\text{Lor}_{1},\text{Euc}_{3}|_{3}}{|\text{Lor}_{1},\text{Lor}_{2}|_{3}}\binom{\overline{2}}{1}_{3}^{\perp}=\frac{|\text{Lor}_{1},\text{Euc}_{3}|_{3}|\text{Euc}_{1},\text{Lor}_{2}|_{3}}{|\text{Lor}_{1},\text{Lor}_{2}|_{3}}=6.\]

Similar to \cref{polynomials}, $\binom{n}{\overline{k}}_{q}^{\perp}$ can be written in terms of the $q$-binomial coefficients and some polynomials, and  $\binom{n}{\overline{k}}_{q}^{\perp}$ is a polynomial in $q$ whose coefficients are half-integers. 
The proof can be obtained by a similar way with \cref{polynomials}.
\begin{proposition}\label{polynomial2}
The expression of $\binom{n}{\overline{k}}_{q}^{\perp}$ is given by the $q$-binomial coefficients as in \textup{\cref{Table: Lorent1}} and \textup{\cref{Table: Lorent2}}. 
Moreover, $\binom{n}{\overline{k}}_{q}^{\perp} \in \frac{1}{2}\mathbb{Z}[q]$.

\begin{table}[H]
\begin{center}
\renewcommand{\arraystretch}{1.7}
\scalebox{0.9}{
\begin{tabular}{c||c|c}
$\mathbf{\binom{n}{\overline{k}}_{q}^{\perp}}$ &  $\mathbf{k}$ \textbf{\textup{is odd}} & $\mathbf{k}$ \textbf{\textup{is even}} \\ \hline \hline
$\mathbf{n}$ \textbf{\textup{is odd}} &  $\frac{1}{2}q^{\frac{k(n-k)}{2}}(q^{\frac{n-k}{2}}-1)\binom{(n-1)/2}{(k-1)/2}_{q^{2}}$ & 
$\frac{1}{2}q^{\frac{k(n-k)}{2}}(q^{\frac{k}{2}}-1)\binom{(n-1)/2}{k/2}_{q^{2}}$ \\ \hline
$\mathbf{n}$ \textbf{\textup{is even}} &  $\frac{1}{2}q^{\frac{k(n-k)-1}{2}}(q^{\frac{n}{2}}-1)\binom{(n-2)/2}{(k-1)/2}_{q^{2}}$ & 
$\frac{1}{2}q^{\frac{k(n-k)}{2}}\frac{(q^{\frac{n-k}{2}}-1)(q^{\frac{k}{2}}-1)}{q^{\frac{n}{2}}+1}\binom{n/2}{k/2}_{q^{2}}$ \\ 
\end{tabular}}.
\caption{The expression of $\binom{n}{\overline{k}}_{q}^{\perp}$ when $q \equiv 1$ (mod $4$)}\label{Table: Lorent1}
\end{center}
\end{table}

\begin{table}[H]
\begin{center}
\renewcommand{\arraystretch}{1.7}
\scalebox{0.9}{
\begin{tabular}{c||c|c}
$\mathbf{\binom{n}{\overline{k}}_{q}^{\perp}}$ &  $\mathbf{k \equiv 1 \pmod{4}}$ & $\mathbf{k \equiv 2 \pmod{4}}$  \\ \hline \hline
$\mathbf{n \equiv 1 \pmod{4}}$ &  $\frac{1}{2}q^{\frac{k(n-k)}{2}}(q^{\frac{n-k}{2}}-1)\binom{(n-1)/2}{(k-1)/2}_{q^{2}}$ & 
$\frac{1}{2}q^{\frac{k(n-k)}{2}}(q^{\frac{k}{2}}+1)\binom{(n-1)/2}{k/2}_{q^{2}}$  \\ \hline
$\mathbf{n \equiv 2 \pmod{4}}$ &  $\frac{1}{2}q^{\frac{k(n-k)-1}{2}}(q^{\frac{n}{2}}+1)\binom{(n-2)/2}{(k-1)/2}_{q^{2}}$ & 
$\frac{1}{2}q^{\frac{k(n-k)}{2}}\frac{(q^{\frac{n-k}{2}}+1)(q^{\frac{k}{2}}-1)}{q^{\frac{n}{2}}-1}\binom{n/2}{k/2}_{q^{2}}$  \\  \hline
$\mathbf{n \equiv 3 \pmod{4}}$ &  $\frac{1}{2}q^{\frac{k(n-k)}{2}}(q^{\frac{n-k}{2}}+1)\binom{(n-1)/2}{(k-1)/2}_{q^{2}}$ & 
$\frac{1}{2}q^{\frac{k(n-k)}{2}}(q^{\frac{k}{2}}+1)\binom{(n-1)/2}{k/2}_{q^{2}}$ \\ \hline
$\mathbf{n \equiv 0 \pmod{4}}$ &  $\frac{1}{2}q^{\frac{k(n-k)-1}{2}}(q^{\frac{n}{2}}-1)\binom{(n-2)/2}{(k-1)/2}_{q^{2}}$ & 
$\frac{1}{2}q^{\frac{k(n-k)}{2}}\frac{(q^{\frac{n-k}{2}}+1)(q^{\frac{k}{2}}+1)}{q^{\frac{n}{2}}+1}\binom{n/2}{k/2}_{q^{2}}$ \\ 
\end{tabular}
}
\end{center}
\end{table}
\smallskip

\begin{table}
\begin{center}
\scalebox{0.9}{
\begin{tabular}{c||c|c}
$\mathbf{\binom{n}{\overline{k}}_{q}^{\perp}}$ &  $\mathbf{k \equiv 3 \pmod{4}}$ & $\mathbf{k \equiv 0 \pmod{4}}$  \\ \hline \hline
$\mathbf{n \equiv 1 \pmod{4}}$ &  $\frac{1}{2}q^{\frac{k(n-k)}{2}}(q^{\frac{n-k}{2}}+1)\binom{(n-1)/2}{(k-1)/2}_{q^{2}}$ & 
$\frac{1}{2}q^{\frac{k(n-k)}{2}}(q^{\frac{k}{2}}-1)\binom{(n-1)/2}{k/2}_{q^{2}}$  \\ \hline
$\mathbf{n \equiv 2 \pmod{4}}$ &  $\frac{1}{2}q^{\frac{k(n-k)-1}{2}}(q^{\frac{n}{2}}+1)\binom{(n-2)/2}{(k-1)/2}_{q^{2}}$ & 
$\frac{1}{2}q^{\frac{k(n-k)}{2}}\frac{(q^{\frac{n-k}{2}}-1)(q^{\frac{k}{2}}+1)}{q^{\frac{n}{2}}-1}\binom{n/2}{k/2}_{q^{2}}$  \\  \hline
$\mathbf{n \equiv 3 \pmod{4}}$ &  $\frac{1}{2}q^{\frac{k(n-k)}{2}}(q^{\frac{n-k}{2}}-1)\binom{(n-1)/2}{(k-1)/2}_{q^{2}}$ & 
$\frac{1}{2}q^{\frac{k(n-k)}{2}}(q^{\frac{k}{2}}-1)\binom{(n-1)/2}{k/2}_{q^{2}}$ \\ \hline
$\mathbf{n \equiv 0 \pmod{4}}$ &  $\frac{1}{2}q^{\frac{k(n-k)-1}{2}}(q^{\frac{n}{2}}-1)\binom{(n-2)/2}{(k-1)/2}_{q^{2}}$ & 
$\frac{1}{2}q^{\frac{k(n-k)}{2}}\frac{(q^{\frac{n-k}{2}}-1)(q^{\frac{k}{2}}-1)}{q^{\frac{n}{2}}+1}\binom{n/2}{k/2}_{q^{2}}$ \\
\end{tabular}}.
\caption{The expression of $\binom{n}{\overline{k}}_{q}^{\perp}$ when $q \equiv 3$ (mod $4$)}\label{Table: Lorent2}
\end{center}
\end{table}
\end{proposition}
Recall that $\binom{n}{k}_{q}$ and $\binom{n}{\overline{k}}_{q}^{\perp}$ are polynomials in $q$ whose coefficients are half-integers. 
The next proposition further shows that $\binom{n}{k}_{q}+\binom{n}{\overline{k}}_{q}^{\perp}$ are in $\mathbb{Z}_{\ge 0}[q]$ except for some cases. 
These exceptional cases coincide with the condition that the limits in \cref{limit} are zero. 

\begin{proposition}\label{prop: nondeg}
The expression of $\binom{n}{k}_{q}^{\perp}+\binom{n}{\overline{k}}_{q}^{\perp}$, the number of all non-degenerate quadratic subspaces of $\F$, is given in \textup{\cref{Table: nondeg1}} and \textup{\cref{Table: nondeg2}}. 
Moreover, $\binom{n}{k}_{q}^{\perp}+\binom{n}{\overline{k}}_{q}^{\perp} \in \mathbb{Z}_{\ge 0}[q]$ except for the case where $n$ is even, $k$ is odd if $q \equiv 1$ \textup{(mod $4$)} or $n \equiv 0 \pmod{4}$, $k$ is odd if $q \equiv 3$ \textup{(mod $4$)}.
\begin{table}[H]
\begin{center}
\renewcommand{\arraystretch}{1.7}
\scalebox{0.9}{
\begin{tabular}{c||c|c}
$\binom{n}{k}_{q}^{\perp}+\binom{n}{\overline{k}}_{q}^{\perp}$ &  $\mathbf{k}$ \textbf{\textup{is odd}} & $\mathbf{k}$ \textbf{\textup{is even}} \\ \hline \hline
$\mathbf{n}$ \textbf{\textup{is odd}} &  $q^{\frac{(k+1)(n-k)}{2}}\binom{(n-1)/2}{(k-1)/2}_{q^{2}}$ & 
$q^{\frac{k(n-k)+k}{2}}\binom{(n-1)/2}{k/2}_{q^{2}}$ \\ \hline
$\mathbf{n}$ \textbf{\textup{is even}} &  $q^{\frac{k(n-k)-1}{2}}(q^{\frac{n}{2}}-1)\binom{(n-2)/2}{(k-1)/2}_{q^{2}}$ & 
$q^{\frac{k(n-k)}{2}}\binom{n/2}{k/2}_{q^{2}}$ \\ 
\end{tabular}}
\caption{The expression of $\binom{n}{k}_{q}^{\perp}+\binom{n}{\overline{k}}_{q}^{\perp}$ when $q \equiv 1$ (mod $4$)}\label{Table: nondeg1}
\end{center}
\end{table}

\begin{table}[H]
\begin{center}
\renewcommand{\arraystretch}{1.7}
\scalebox{0.9}{
\begin{tabular}{c||c|c}
$\binom{n}{k}_{q}^{\perp}+\binom{n}{\overline{k}}_{q}^{\perp}$ &  $\mathbf{k}$ \textbf{\textup{is odd}} & $\mathbf{k}$ \textbf{\textup{is even}}  \\ \hline \hline
$\mathbf{n}$ \textbf{\textup{is odd}} &  $q^{\frac{(k+1)(n-k)}{2}}\binom{(n-1)/2}{(k-1)/2}_{q^{2}}$ & 
$q^{\frac{k(n-k)+k}{2}}\binom{(n-1)/2}{k/2}_{q^{2}}$  \\ \hline
$\mathbf{n \equiv 2 \pmod{4}}$ &  $q^{\frac{k(n-k)-1}{2}}(q^{\frac{n}{2}}+1)\binom{(n-2)/2}{(k-1)/2}_{q^{2}}$ &
\multirow{2}{*}{$q^{\frac{k(n-k)}{2}}\binom{n/2}{k/2}_{q^{2}}$ } \\ 
\cline{1-2} 
$\mathbf{n \equiv 0 \pmod{4}}$ &  $q^{\frac{k(n-k)-1}{2}}(q^{\frac{n}{2}}-1)\binom{(n-2)/2}{(k-1)/2}_{q^{2}}$ &  \\ 
\end{tabular}}
\caption{The expression of $\binom{n}{k}_{q}^{\perp}+\binom{n}{\overline{k}}_{q}^{\perp}$ when $q \equiv 3$ (mod $4$)}\label{Table: nondeg2}
\end{center}
\end{table}
\end{proposition}

\begin{corollary}\label{cor: limit}
For any $1 \le k<n$, we have 
\begin{enumerate}
    \item[\textup{(1)}] If $q \equiv 1$ \textup{(mod $4$)}, $\lim_{q \to 1}\binom{n}{\overline{k}}_{q}^{\perp}=0$,
    \item[\textup{(2)}] If $q \equiv 1$ \textup{(mod $4$)}, 
    \[\lim_{q \to -1}\binom{n}{\overline{k}}_{q}^{\perp}=\begin{cases}
    ~~\binom{\left \lfloor n/2 \right \rfloor}{\left \lfloor k/2 \right \rfloor}
 &  \text{ if } n \equiv 1 ~(\textup{mod }4),~k \equiv 2 ~(\textup{mod }4) \text{ or } k \equiv 3 ~(\textup{mod }4),  \\
 & \quad  n \equiv 3 ~(\textup{mod }4),~k \equiv 1 ~(\textup{mod }4)\text{ or }k \equiv 2 ~(\textup{mod }4), \\
 & \quad n \equiv 0 ~(\textup{mod }4),~k \equiv 2 ~(\textup{mod }4),  \\
~~\binom{ n/2  -1}{\left \lfloor k/2 \right \rfloor} & \text{ if } n \equiv 2 ~(\textup{mod }4),~k \equiv 1 ~(\textup{mod }4),\text{ or }k \equiv 2 ~(\textup{mod }4) \text{ or } k \equiv 3 ~(\textup{mod }4), \\
~~\binom{ n/2  -1}{ k/2  -1} & \text{ if } n \equiv 2 ~(\textup{mod }4),~k \equiv 0 ~(\textup{mod }4), \\
~~0 & \text{ otherwise } .
\end{cases}\]
    \item[\textup{(3)}] If $q \equiv 3$ \textup{(mod $4$)}, $\lim_{q \to 1}\binom{n}{\overline{k}}_{q}^{\perp}$ is the same as $\lim_{q \to -1}\binom{n}{\overline{k}}_{q}^{\perp}$ when $q \equiv 1$ (\textup{mod} $4$), 
    \item[\textup{(4)}] If $q \equiv 3$ \textup{(mod $4$)}, $\lim_{q \to -1}\binom{n}{\overline{k}}_{q}^{\perp}=0$.
\end{enumerate}
\end{corollary}

By \cref{limit} and \cref{cor: limit}, the limit of the number of non-degenerate subspaces of $\F$ has the same combinatorial phenomenon with the Euclidean binomial coefficient in \cref{symm}. 
\begin{itemize}
    \item If $q \equiv 1$ (\textup{mod} $4$), $\lim_{q \to 1}\big ( \binom{n}{k}_{q}^{\perp}+\binom{n}{\overline{k}}_{q}^{\perp} \big)$ is the number of symmetric $k$-sets in $\mathbb{Z}/(n+1)\mathbb{Z}$.
    \item If $q \equiv 3$ (\textup{mod} $4$), $\lim_{q \to -1}\big( \binom{n}{k}_{q}^{\perp}+\binom{n}{\overline{k}}_{q}^{\perp}\big)$ is the number of symmetric $k$-sets in $\mathbb{Z}/(n+1)\mathbb{Z}$. 
\end{itemize}

\subsection{Combinatorial properties of $LO_{n}(q)$}

We construct a poset $LO_{n}(q):=(\mathcal{L},\subset)$, where $\mathcal{L}$ is the set of all Lorentzian subspaces of $(\mathbb{F}_{q}^{n},\text{Euc}_{n})$ and call it the \textbf{Lorentzian poset}. We also do not consider the empty set to be a subspace, but we consider the zero space as the least element and $(\mathbb{F}_{q}^{n},\text{Euc}_{n})$ as the maximal element of the Lorentzian poset. 
Recall that $|\text{Lor}_{k},\text{Euc}_{n}|_{q}$ is the number of $k$-dimensional Lorentzian subspaces of $(\mathbb{F}_{q}^{n},\text{Euc}_{n})$.
In the following, we ask the same questions as for the Euclidean poset.

\begin{proposition}\label{prop: LO1}
For any $1\le k<n$, the Lorentzian poset $LO_{n}(q)$ is rank-symmetric. In other words, $|\textup{Lor}_{k},\textup{Euc}_{n}|_{q}=|\textup{Lor}_{n-k},\textup{Euc}_{n}|_{q}$.
\end{proposition}

\begin{proof}
Let $L_{1}$ be a $k$-dimensional Lorentzian subspace of $\F$. Then, there is a $(n-k)$-dimensional Lorentzian subspace $L_{2}$ of $\F$ such that
\[L_{1}\oplus L_{2}=\F=L_{1}\oplus L_{1}^{\perp}.\] 
By Witt's cancellation theorem, we obtain $L_{2}=L_{1}^{\perp}$. Since taking $\perp$ is bijective, the result holds. 
\end{proof}

\begin{proposition}\label{prop: LO2} The Lorentzian poset $LO_{n}(q)$ is rank-unimodal. i.e. For any $n>0$, there is $j$ such that 
\[ \binom{n}{\overline{1}}_{q}^{\perp} \leq \binom{n}{\overline{2}}_{q}^{\perp} \leq \cdots \leq \binom{n}{\overline{j}}_{q}^{\perp} \geq \cdots \geq \binom{n}{\overline{(n-1)}}_{q}^{\perp}. \]
\end{proposition}

\begin{proof}
We consider the following:
\[M:=\frac{\binom{n}{\overline{k}}_{q}^{\perp}}{\binom{n}{\overline{k-1}}_{q}^{\perp}}=\frac{|\text{Lor}_{k},\text{Euc}_{n}|_{q}}{|\text{Lor}_{k-1},\text{Euc}_{n}|_{q}}=\frac{\frac{|\text{Lor}_{1},\text{Euc}_{n}|_{q}}{|\text{Lor}_{1},\text{Euc}_{k}|_{q}}}{\frac{|\text{Lor}_{1},\text{Euc}_{n}|_{q}}{|\text{Lor}_{1},\text{Lor}_{k-1}|_{q}}}
\frac{\binom{\overline{n-1}}{k-1}_{q}^{\perp}}{\binom{\overline{n-1}}{k-2}_{q}^{\perp}}=\frac{|\text{Lor}_{1},\text{Lor}_{k-1}|_{q}}{|\text{Lor}_{1},\text{Lor}_{k}|_{q}}\frac{|\text{Euc}_{1},\text{Lor}_{n-k+1}|_{q}}{|\text{Euc}_{1},\text{Euc}_{k-1}|_{q}}.\]
We note that $M \ge 1$ if $n-k+1 > k$. i.e. $\frac{n+1}{2} > k$. 
On the other hand, $M \le 1$ if $n-k+1 < k$. i.e., $\frac{n+1}{2} < k$. If $n$ is odd, $\binom{n}{\overline{\frac{n-1}{2}}}_{q}^{\perp}=\binom{n}{\overline{\frac{n+1}{2}}}_{q}^{\perp}$ since $LO_{n}(q)$ is rank-symmetric. Thus we set $j=(n-1)/2$ or $j=(n+1)/2$.
If $n$ is even, let us choose $j=n/2$. It follows that $LO_{n}(q)$ is rank-unimodal.
\end{proof}

\section{Future research}\label{Future}

(1) Given a set of combinatorial objects $\mathcal{C}$, a \textbf{combinatorial statistic} is a function from $\mathcal{C}$ to $\mathbb{Z}_{\ge 0}$. For example, let $\mathcal{C}$ be the set of Young diagrams of the size $k \times (n-k)$, and define a function $f: \mathcal{C} \rightarrow \mathbb{Z}_{\ge 0}$ such that $f(\lambda)\text{ is the sum of the parts of } \lambda$.
Note that $q$-binomial coefficients $\binom{n}{k}_{q}$ has a combinatorial statistic $f: \mathcal{C} \rightarrow \mathbb{Z}_{\geq 0}$ as follows:
\[\binom{n}{k}_{q}=\sum_{\lambda \in \mathcal{C}}q^{f(\lambda)}.\]
Recall that $\binom{n}{k}_{q}^{\perp}+\binom{n}{\overline{k}}_{q}^{\perp}$ is the number of non-degenerate $k$-dimensional subspaces of $(\mathbb{F}_{q}^{n},\textup{Euc}_{n})$.
\cref{prop: nondeg} shows that the coefficients of $\binom{n}{k}_{q}^{\perp}+\binom{n}{\overline{k}}_{q}^{\perp}$ are nonnegative except for two cases. 
Based on this observation, it is natural to ask the following question. 

\begin{question}
Are there any combinatorial statistics for $\binom{n}{k}_{q}^{\perp}+\binom{n}{\overline{k}}_{q}^{\perp}$?
\end{question}

Additionally, one could ask if there are other combinatorial ways to formulate $\binom{n}{k}_{q}^{\perp}$ even though $\binom{n}{k}_{q}^{\perp}$ is a polynomial in $q$ which contains non-positive coefficients for several cases. 

\smallskip
(2) An \textbf{antichain} in $P$ is a subset $A$ of $P$ such that no two elements are comparable. A graded poset $P$ of rank $n$ has the \textbf{Sperner property} if \[\text{max}\left \{ |A|~|~A\text{ is an antichain of }P\right \}=\text{max}\left \{ |P_{k}|~|~0\leq k \leq n \right \}\]
The motivation for the Sperner property comes from counting the largest antichain in a Boolean algebra. 
The more history can be found in \cite{St1}.
Note that the Boolean algebra $B_{n}$ and the poset $L_{n}(q)$ constructed by all subspaces of $\mathbb{F}_{q}^{n}$ are Sperner. 
We ask the same question to the posets $E_{n}(q)$ and $LO_{n}(q)$.

\begin{question}
Are the Euclidean poset $E_{n}(q)$ and the Lorentzian poset $LO_{n}(q)$ Sperner?
\end{question}

For example, a sage code indicates that the Euclidean poset $E_5(3)$ is Sperner even though it is not a lattice. 
The main difficulty to solve this question follows the type of the intersection between quadratic spaces. 
The same strategy using order-matching maps from \cite{St1} does not work since the intersection between two $k$-dimensional Euclidean subspaces of $(\mathbb{F}_{q}^{n},\textup{Euc}_{n})$ is not Euclidean in general.

\medskip
\noindent \textbf{The conflict of interest statement.}
\noindent On behalf of all authors, the corresponding author states that there is no conflict of interest.

\end{document}